\documentclass[12pt, oneside]{amsart}

\usepackage[utf8]{inputenc}
\usepackage[T1]{fontenc}

\usepackage{amssymb}
\usepackage{amsthm}
\usepackage{amsmath}
\usepackage[only,Yup]{stmaryrd} %%% PARA O FIBRADO DE MOLINO
\usepackage{mathrsfs}    %%% PARA PSEUDOGRUPOS
\usepackage[new]{old-arrows}   %%% PARA SETAS DE INJEÇÕES
\usepackage{comment}
\usepackage{soul}
\usepackage{cancel}
\usepackage{wasysym}
\usepackage{accents}

\usepackage[a4paper,left=2cm,top=3cm,right=2cm,bottom=2cm]{geometry}

\usepackage[all,cmtip]{xy}

\usepackage{enumerate}

\usepackage{indentfirst}

\usepackage{hyperref}
\hypersetup{colorlinks=true,allcolors=black}  %%% COR DO LINK

\usepackage{pstricks}
\usepackage{graphicx}

\usepackage{multicol}

\mathchardef\ordinarycolon\mathcode`\:
\mathcode`\:=\string"8000
\begingroup \catcode`\:=\active
  \gdef:{\mathrel{\mathop\ordinarycolon}}
\endgroup

\begin{document}

\title{A Hadamard theorem in transversely affine geometry with applications to affine orbifolds}

\author{Francisco C.~Caramello Jr.}
\address{Departamento de Matemática, Universidade Federal de Santa Catarina, R. Eng. Agr. Andrei Cristian Ferreira, 88040-900, Florianópolis - SC, Brazil}
\email{francisco.caramello@ufsc.br}

\author{H. A. Puel Martins}
\address{Departamento de Matemática, Universidade Federal de Santa Catarina, R. Eng. Agr. Andrei Cristian Ferreira, 88040-900, Florianópolis - SC, Brazil}
\email{henrique.martins@posgrad.ufsc.br}

\author{Ivan P. Costa e Silva}
\address{Departamento de Matemática, Universidade Federal de Santa Catarina, R. Eng. Agr. Andrei Cristian Ferreira, 88040-900, Florianópolis - SC, Brazil}
\email{pontual.ivan@ufsc.br}


%%% AMBIENTES %%%
\newenvironment{proofoutline}{\proof[Proof outline]}{\endproof}
\newenvironment{proofcomment}{\proof[Comment on the proof]}{\endproof}
\theoremstyle{definition}
\newtheorem{example}{Example}[section]
\newtheorem{definition}[example]{Definition}
\newtheorem{remark}[example]{Remark}
\theoremstyle{plain}
\newtheorem{proposition}[example]{Proposition}
\newtheorem{theorem}[example]{Theorem}
\newtheorem{lemma}[example]{Lemma}
\newtheorem{corollary}[example]{Corollary}
\newtheorem{claim}[example]{Claim}
\newtheorem{conjecture}[example]{Conjecture}
\newtheorem{thmx}{Theorem}
\renewcommand{\thethmx}{\Alph{thmx}} % "letter-numbered" theorems
\newtheorem{corx}[thmx]{Corollary}
\renewcommand{\thecorx}{\Alph{corx}} % "letter-numbered" corollaries

%%% MACROS %%%
\newcommand{\dif}[0]{\mathrm{d}}
\newcommand{\od}[2]{\frac{\dif #1}{\dif #2}}
\newcommand{\pd}[2]{\frac{\partial #1}{\partial #2}}
\newcommand{\dcov}[2]{\frac{\nabla #1}{\dif #2}}
\newcommand{\proin}[2]{\left\langle #1, #2 \right\rangle}
\newcommand{\f}[0]{\mathcal{F}}
\newcommand{\g}[0]{\mathcal{G}}
\newcommand{\metric}{\ensuremath{\mathrm{g}}}
\newcommand{\qcd}{\begin{flushright} $\Box$ \end{flushright}}
\begin{abstract}
We introduce and investigate a novel notion of transversely affine foliation, comparing and contrasting it to the previous ones in the literature. We then use it to give an extension of the classic Hadamard’s theorem from Riemannian geometry to this setting. Our main result is a transversely affine version of a well-known ``Hadamard-like'' theorem by J. Hebda for Riemannian foliations. Alternatively, our result can be viewed as a foliation-theoretic analogue of the Hadamard's theorem for affine manifolds proven by Beem and Parker. Namely, we show that under the transverse analogs of pseudoconvexity and disprisonment for the family of geodesics in the transverse affine geometry, together with an absence of transverse conjugate points, the universal cover of a manifold endowed with a transversely affine foliation whose leaves are compact and with finite holonomy is diffeomorphic to the product of a contractible manifold with the universal cover of a leaf. This also leads to a Beem--Parker-type Hadamard-like theorem for affine orbifolds.
\end{abstract}

\maketitle
\setcounter{tocdepth}{1}
\tableofcontents

\section{Introduction}

%The goal of this paper is to obtain a version of the affine Hadamard's theorem, in the context of transverse geometry of regular foliations.

Let $(M,h)$ be a connected Riemannian manifold. The classic Hadamard's theorem in Riemannian geometry \cite{oneillbook,petersen} states that the exponential map $\exp^h_p:T_pM \rightarrow M$ at any point $p\in M$ is a (universal) covering map, provided $(M,h)$ is geodesically complete and its sectional curvature is nonpositive. Since this result implies such a nice control over the underlying manifold's topology, it naturally motivates attempts of generalizations to a broader class of geometries --- say  semi-Riemannian or affine manifolds --- where the notion of an exponential map still makes sense. 

At first, such attempts would seem misguided, because geodesic completeness and nonpositivity of sectional curvatures, which are natural in Riemannian geometry (see. e.g., \cite{morrow,NO}) are much less so in these broader contexts. For example, it is well-known that semi-Riemannian metrics do not, in general, define any distance function, and no analogue of the Hopf-Rinow theorem exists \cite{beem,oneillbook}. In particular, compactness of the underlying manifold does not imply either geodesic completeness or geodesic connectedness in this context \cite[Example 7.16, Exercise 9.12]{oneillbook}. In addition, in the semi-Riemannian context it is well-known that both geodesic completeness and how the existence of conjugate points is affected by sectional curvatures depend on the so-called \textit{causal character} of the geodesics \cite{beemincomplete}, and furthermore many of the most interesting geometric examples in Lorentzian geometry --- arguably the most important semi-Riemannian type of geometry apart from the Riemannian one --- are geodesically incomplete \cite[Ch. 5]{hawking-ellis}. %Even compact and geodesically complete affine manifolds may fail to be geodesically connected. 

\begin{comment}
Yet, upon inspecting more closely, the standard proof of Hadamard's theorem (conf. \cite[Thm. 7.29, Cor. 7.29, Thm. 10.22]{oneillbook} it becomes clear that $(1)$ the condition of nonpositivity of sectional curvature is only there to ensure the absence of conjugate points along geodesics, or more precisely that each $\exp_p$ is a local diffeomorphism, and $(2)$ that there \textcolor{blue}{(?)} to ensure a suitable path-lifting property for the exponential maps. Upon that notion, there have been a number of alternative results in these broader contexts that still can be regard as ``Hadamard-like'' theorems. 
\end{comment}

These issues notwithstanding, there have been a number of concrete alternative results in these broader contexts that still can be regard as ``Hadamard-like'' theorems. For example, Flaherty \cite{flaherty}, studying the particular case of Lorentzian manifolds, applied the hypotheses of geodesic completeness and sign for the sectional curvatures only in timelike directions, though at the cost of introducing the fairly strong hypothesis of \textit{future 1-connectedness}. i.e., that any two timelike curves starting and ending at the same points can be homotopically deformed through timelike intermediate curves.

\begin{theorem}[Flaherty]
Let $M$ be a time-oriented, causally future-complete Lorentzian manifold which is also future 1-connected. Further suppose that the sectional curvature is nonpositive for all timelike planes. Then the exponential map regularly embeds the future timecone at each point into $M$.
\end{theorem}
\noindent Observe that having hypotheses only in the timelike directions entails that the result affects only timecones, a kind of ``partial covering map'' result. 

A different approach was taken by Beem and Parker \cite{beemparker}, who obtained a result applicable to the broader class of affine manifolds $(M,\nabla)$. As hypotheses, they directly impose the absence of pairs of conjugate points along all geodesics --- a property that follows from nonpositivity of sectional curvatures in the Riemannian case --- and ensure a suitable path-lifting property in lieu of completeness, via the concepts of \textit{disprisonment} and \textit{pseudoconvexity} of the family of all geodesics of $(M,\nabla)$. These are standard geometric constraints meant to control the behavior of the geodesic flow and other dynamical systems. (See the references in \cite{beemparker} for an extensive list of classic applications). However, since (a transversal version of) disprisonment and pseudoconvexity will turn out to be important for us in this paper, we briefly recall the latter notions here. 

\begin{definition}
Let $(M,\nabla)$ be an affine manifold and consider a family $\mathcal{C}$ of non-constant $\nabla$-geodesics. We say that $\mathcal{C}$ is
\begin{itemize}
    \item \textit{pseudoconvex} if for every compact set $K\subset M$ there is another compact set $K^*\subset M$ such that if any segment of $\gamma\in\mathcal{C}$ has its endpoints in $K$, then that segment is entirely contained in $K^*$;
    \item \textit{disprisoning} if given the maximal extension $\gamma:(a,b)\rightarrow M$ of a geodesic in $\mathcal{C}$ and any $t_0\in(a,b)$, both $\gamma|_{(a,t_0]}$ and $\gamma|_{[t_0, b)}$ fail to have compact closure. 
\end{itemize}
\end{definition} 
\noindent If the class $\mathcal{C}$ of \textit{all} geodesics is pseudoconvex and disprisoning, then we say that $(M,\nabla)$ itself is pseudoconvex and disprisoning.

\begin{theorem}[Beem--Parker]
Let $(M^n,\nabla)$ be both pseudoconvex and disprisoning, and suppose that no inextendible geodesic therein has a pair of conjugate points. Then $(M,\nabla)$ is geodesically connected and, for each $p\in M$, the exponential map $\exp_p:\mathcal{D}_p\subset T_pM\rightarrow M$ is a smooth covering map. In particular, the universal covering manifold of $M$ is diffeomorphic to $\mathbb{R}^n$.  
\end{theorem}

Our goal in this paper is to give a foliation-theoretic analogue of the Beem--Parker result described above. A key antecedent to our work is a well-known result by Hebda \cite{hebda}, who obtained a \textit{transverse} version of Hadamard's theorem for \emph{Riemannian} foliations, \textit{i.e.} a foliated manifold endowed with a Riemannian \textit{bundle-like metric}, which is a Riemannian metric $h$ such that $h(X,Y)$ is a basic function whenever $X,Y$ are projectable fields $g$-orthogonal to $\mathcal{F}$. 

\begin{theorem}[Hebda]
Let $(M,\mathcal{F})$ be a Riemannian, codimension $q$ foliation with a complete bundle-like metric $h$. If the sectional curvature function of $h$ is nonpositive, then the universal cover $\tilde{M}$ of $M$ fibers over a complete simply-connected Riemannian $q$-manifold $N$ of nonpositive curvature. Moreover, each fiber in this fibration is diffeomorphic to the universal covering of a leaf of $\mathcal{F}$. In particular, because the fibration is trivial since $N\simeq \mathbb{R}^q$ is contractible, $\tilde{M}$ is diffeomorphic to $\tilde{L}\times N$.
\end{theorem}
\noindent Observe that Hebda's result reduces to the standard Hadamard's theorem for a Riemannian manifold $(M,h)$ when the foliation is zero-dimensional, that is, when it is the trivial foliation of $M$ by its points. 

Our work can be alternatively construed as an affine-geometric version of Hebda's result. However, as we shall see in more detail below, there are in the literature many alternative, not necessarily equivalent notions of what a \textit{transversely affine foliation} should be. Therefore, we divide the paper in two parts. In the first part, we give an extended discussion of, and develop a new machinery to deal with, affine aspects of the transverse geometry of foliations. In particular, since we are interested in (a suitable notion) of transverse geodesics, instead of dealing with a connection on the normal bundle $\nu\mathcal{F}$ of the foliation $\f$ we introduce a novel concept of a connection on the underlying manifold itself with suitable ``bundle-like'' properties. Far from being a mere transcription of already known data into another ambiance, the introduction of the so-called \textit{transverse affine connections} and \textit{transverse affine structures} allow us to deal with perfect transverse analogues of classical objects of affine geometry such as parallel transport, geodesics, curvature and Jacobi fields. Since to the best of our knowledge much of this material is new, we discuss it in great detail here. 

Our main result, which we shall prove towards the end of the paper, is a mixture of the ``Hadamard-like'' Beem-Parker result we mentioned above with a conclusion resembling Hebda's:

\begin{thmx}
Let $(M,\mathcal{F})$ be a foliation whose leaves are all compact and such that $M/\mathcal{F}$ is Hausdorff. Suppose $(M,\mathcal{F}$) is endowed with a transverse affine structure such that the family of all non-vertical transverse-geodesics is transversely pseudoconvex, transversely disprisoning and without transverse conjugate points. Then the universal cover $\tilde{M}$ of $M$ splits diffeomorphically as a product $\tilde{L}\times B$, where $B$ is a contractible manifold and $\tilde{L}$ is the universal cover of some leaf $L$ of $\mathcal{F}$, so that the lifted foliation $\tilde{\f}$ is given by the fibers of $\tilde{M}\to B$.
\end{thmx}

In addition, if we abide by the philosophy -- common in the foliation-theoretic lore -- of ``transverse-geometry-equals-leaf-space-geometry'', then this part of our work can also be understood as a high-regularity technology to do geometry in a low-regularity setting. As a key application of the latter philosophy, we consider \textit{orbifolds}, which are well-known generalizations of manifolds (see, e.g., \cite{caramello3} for an introduction with extensive references) but can also be naturally realized as leaf spaces of foliations with compact leaves. Our results here allow us to consider orbifolds endowed with an affine structure, and we then obtain a generalization of the Beem-Parker result to this context. 

\begin{corx}[Affine Hadamard theorem for orbifolds]\label{corx}
Let $(\mathcal{O},\nabla)$ be an affine orbifold without conjugate points and such that the family of its geodesics is pseudoconvex and disprisoning. Then $\mathcal{O}$ is a good (developable) orbifold whose universal covering is a contractible manifold. 
\end{corx}

The rest of the paper is organized as follows. In section \ref{sec2}, we recall a number of basic foliation-theoretic concepts we shall use, mainly to establish notation and terminology. In section \ref{sec3} we give an extended description of the main different notions of transverse affine geometry on foliations found in the literature and their mutual relationship, and introduce our own, showing how they relate to previous ones. In section \ref{section: geodesics} we discuss geodesic-related notions, especially a notion of \textit{transverse-geodesic}, or geodesic for the transverse geometry, with associated notions of Jacobi fields and conjugate points thereon. Finally, in section \ref{sec5} we give the technical details and proof of the main theorem described above, together with its corollary for orbifolds.  

\section{Preliminaries}\label{sec2}

We assume the reader is familiar with the basic elements of foliation theory as given in textbooks such as \cite{candel,mrcun,molino}, but we briefly review a few basic pertinent facts mainly to establish the terminology we shall use. Throughout this work, $(M^n,\f^q)$ will denote a \textit{foliated manifold}, \textit{i.e.}, $M$ is a smooth $n$-manifold\footnote{Here, for simplicity \textit{smooth} always means $C^\infty$, although many of the results discussed here may be pursued in lower regularity, and we assume that $M$ is Hausdorff and second-countable.} and $\f$ is a partition of $M$ by $q$-codimensional ($0\leq q \leq n$), immersed, connected, smooth submanifolds of $\f$ called the \textit{leaves} of the foliation\footnote{In particular, we deal here only with \emph{regular} foliations $\f$, which means that all the leaves of $\f$ have the same dimension.}. Since not every partition of a manifold by immersed submanifolds gives a foliation, some added conditions are necessary. Indeed, a foliation $\mathcal{F}$ can be defined in many equivalent ways, but particularly convenient for us here is the definition via a \textit{Haefliger cocycle} $(\mathcal{U}_\lambda, s_\lambda, \gamma_{\lambda\mu})_{\lambda,\mu}$, where $\{\mathcal{U}_\lambda\}$ is an open cover of $M$, each $s_\lambda$ is a smooth submersion from $\mathcal{U}_\lambda$ onto an open subset of $\mathbb{R}^q$ and for each $\lambda, \mu$ such that $\mathcal{U}_\lambda\cap\mathcal{U}_\mu\neq\emptyset$ we have that $\gamma_{\lambda\mu}:s_\mu(\mathcal{U}_\lambda\cap\mathcal{U}_\mu)\rightarrow s_\lambda(\mathcal{U}_\lambda\cap\mathcal{U}_\mu)$ is the unique diffeomorphism satisfying $$\gamma_{\lambda\mu}\circ s_\mu|_{\mathcal{U}_\lambda\cap\mathcal{U}_\mu}=s_\lambda|_{\mathcal{U}_\lambda\cap\mathcal{U}_\mu}.$$ In this context, each (nonempty) connected component of a fiber $s_\lambda^{-1}(c)$ ($c\in \mathbb{R}^q$) of any local submersions $s_\lambda$ is an embedded $q$-codimensional submanifold of $M$ called a \textit{plaque}, and we can introduce an equivalence relation $\sim$ on $M$ by demanding that, $\forall x,y \in M$, $x\sim y$ if and only if either $x,y$ are on the same plaque or else there exists some finite sequence $\sigma_1, \ldots,\sigma_k$ of plaques 
with $k\geq 2$ and $\sigma _i\cap \sigma_{i+i} \neq \emptyset$ for $i\in \{1, \ldots, k-1\}$ so that $x\in \sigma_1, y\in \sigma_k$. The equivalence classes via this equivalence relation are precisely the leaves of $\f$, and two Haefliger cocycles are also said to be \textit{equivalent} when they define the same foliation in this way. In particular, for each leaf $L\in \f$, each $L\cap \mathcal{U}_\lambda$ is a countable union of plaques. The quotient space $M/\f$ via this equivalence relation is the \textit{leaf space} of a foliation. 

It is evident from the previous description that given any smooth submersion $F: M\rightarrow N$, the connected components of its fibers are the leaves of a $(\dim N)$-codimensional foliation ($(M,F,Id_{F(N)}) $ defines a Haefliger cocycle) of $N$. Such a foliation is called \textit{simple}, and gives --- via Haefliger cocycles --- the ``local model'' for a general foliation. Hence we say that a foliation is \textit{locally defined} by submersions.

%Given a Haefliger cocycle $(\mathcal{U}_\lambda, s_\lambda, \gamma_{\lambda\mu})_{\lambda,\mu}$, the local diffeomorphisms $\gamma_{\lambda\mu}$ put together define a pseudogroup of transformations on $\mathbb{R}^q$, and the germs of these local diffeomorphisms at points of their domains define a groupoid structure, the so-called \textit{holonomy groupoid} of $\mathcal{F}$. The geometric information encoded in this groupoid can be realized via germs of diffeomorphisms between local transversals at two points of a leaf (see \cite[Sec. 2.3 ]{candel} for more details), by identifying such local transversals with open subsets of $\mathbb{R}^q$. Any such diffeomorphism whose germ is the holonomy groupoid will be called a \textit{holonomy transformation}.

Given a Haefliger cocycle $(\mathcal{U}_\lambda, s_\lambda, \gamma_{\lambda\mu})_{\lambda,\mu}$, the collection $\{\gamma_{\lambda\mu}\}$ defines a pseudogroup $\mathscr{H}$ of local diffeomorphisms on $\mathbb{R}^q$, called the \textit{holonomy pseudogroup} of that cocycle (an equivalent cocycle will lead to a differentiable equivalent pseudogroup, see e.g. \cite{salem}). Choosing sections for the submersions $s_\lambda$ one defines a \textit{total transversal} for $\f$: a (in general disconnected) $q$-dimensional submanifold $S=\bigsqcup_\lambda S_\lambda$ that transversely intersects every leaf of $\f$ at least once --- where each \textit{local transversal} $S_\lambda$ is the image of the chosen section for $s_\lambda$. After identifying the local quotients $s_\lambda(\mathcal{U}_\lambda)$ with the local transversals $S_\lambda$, the geometric information encoded in $\mathscr{H}$ can then be seen as the local diffeomorphisms of $S$ one gets by ``sliding along the leaves'' of $\f$, thus illustrating that $\mathscr{H}$ encodes the recurrence of the leaves (see \cite[Section 1.7]{molino} or \cite{alex4} for details). Any such diffeomorphism will be called a \textit{holonomy transformation}. Still in this vein, the \textit{holonomy group} $\mathrm{Hol}_x(L)$ of a leaf $\f$ at $x\in L\cap S$ is defined as the collection of germs of holonomy transformations that fix $x$.

Since the leaves of $\f$ are immersed submanifolds, we can gather all their tangent spaces into a rank $n-q$ distribution on $M$, i.e., a vector subbundle $T\mathcal{F}$ of $TM$, called the \textit{tangent} or \textit{vertical distribution} (of $\mathcal{F}$). Elements $v\in T\f$ are \textit{vertical vectors}. Throughout this work we will interchangeably denote either by $\Gamma(T\f)$ or by $\mathfrak{X}(\f)$ the space of the smooth sections of the tangent distribution, and call its elements \textit{vertical vector fields}. It is clear that the vertical distribution is \textit{involutive} in the sense that $[V,W] \in \mathfrak{X}(\f)$ whenever $V,W \in \mathfrak{X}(\f)$; in other words, $\mathfrak{X}(\f)$ is a Lie subalgebra of $\mathfrak{X}(M)$. The \textit{normal bundle} of the foliated manifold $(M, \mathcal{F})$ is the rank $q$ vector bundle given by $\nu\mathcal{F}:=TM/T\mathcal{F}$, where the quotient is taken fiberwise. If we look at the respective spaces of sections of these bundles, we have the following short exact sequence of $C^\infty(M)$-modules: $$
0\rightarrow\Gamma(T\mathcal{F})\rightarrow\Gamma(TM)\rightarrow\Gamma(\nu\mathcal{F})\rightarrow 0.
$$
A rank $q$ distribution $\mathcal{H} \subset TM$ such that $T\f \oplus \mathcal{H}=TM$ is said to be \textit{horizontal}. Clearly, horizontal distributions always exist, though not canonically: if we pick any Riemannian metric $g$ on $M$, defining $\mathcal{H}_x: (T\f_x)^\perp$ for each $x\in M$ gives one such distribution. The $C^\infty(M)$-module $\Gamma(\mathcal{H})$ of the sections of any horizontal distribution $\mathcal{H}$ gives a splitting of the previous exact sequence, since we have $\mathcal{H}\simeq \nu\f$ as vector bundles. Whenever a horizontal distribution $\mathcal{H}$ has been chosen, we refer to vectors $v\in \mathcal{H}$ as \textit{horizontal}; elements $X\in \Gamma(\mathcal{H})$ will accordingly be called \textit{horizontal vector fields} and a smooth (or piecewise smooth) curve $\gamma:I\subset \mathbb{R}\rightarrow M$ is also said to be horizontal if its tangent vectors are all horizontal, i.e, $\gamma'(t) \in \mathcal{H}_{\gamma(t)}, \forall t \in I$.

A function $f: M\rightarrow \mathbb{R}$ is called \textit{basic (with respect to $\mathcal{F}$)} (or \textit{$\f$-basic}) if it is constant on the leaves of $\mathcal{F}$. There is a one-to-one correspondence between continuous basic functions on $M$ and continuous real-valued functions on the leaf space $M/\f$, and $f\in C^\infty(M)$ is basic if and only if $Vf \equiv 0$ for any $V \in \mathfrak{X}(\f)$. The set of all smooth basic functions of $(M,\mathcal{F})$, denoted here by $C^\infty(\mathcal{F})$, is an $\mathbb{R}$-subalgebra of $C^\infty(M)$. A vector field $X\in\mathfrak{X}(M)$ is called \textit{projectable} (or \textit{foliated}) if $\mathcal{L}_VX=[V,X]$ is a vertical vector field for any $V\in\Gamma(T\mathcal{F})$. The set of all foliated vector fields is denoted by $\mathfrak{L}(\mathcal{F})$, and it is both a $C^\infty(\mathcal{F})$-module and a Lie subalgebra of $\mathfrak{X}(M)$ containing the vertical vector fields. Since $\Gamma(T\mathcal{F})$ is clearly an ideal of this Lie algebra, we can take the quotient and define $\mathfrak{l}(\mathcal{F}):=\mathfrak{L}(\mathcal{F})/\Gamma(T\mathcal{F})$, which is the set of \textit{transverse vector fields} of $\f$. It inherits the structure of a Lie algebra and of a $C^\infty(\mathcal{F})$-module from $\mathfrak{L}(\mathcal{F})$. Observe also that $\mathfrak{l}(\f)$ is a vector subspace of $\Gamma(\nu \f)$. 

Transverse vector fields can be alternatively characterized as follows.  Given $X,X' \in \mathfrak{X}(M)$, we introduce an equivalence relation 
$$X\sim X' \Longleftrightarrow X-X'\in \mathfrak{X}(\f),$$
and denote the equivalence class of a smooth vector field $X$ via this relation by $\overline{X}$. These classes can be naturally identified with sections of the normal bundle: given any $\sigma \in \Gamma(\nu \f)$, there exists some $Z\in \mathfrak{X}(M)$ for which $\sigma = \overline{Z}$. Observe also that via this identification, $\mathfrak{l}(\f)$ is a vector subspace of $\Gamma(\nu \f)$. Next we can define, for each $X\in \mathfrak{L}(\f)$, 
\begin{equation} \label{lieder}
\mathcal{L}_X\sigma := \overline{[X,Z]}.
\end{equation}
(Here and hereafter $\mathcal{L}_X$ denotes the Lie derivative along $X\in \mathfrak{X}(M)$ of whatever the respective geometric object appearing after it.) This gives a well-defined section of $\nu \f$, and we say that  $\sigma \in \Gamma(\nu \f)$ is \textit{holonomy-invariant (with respect to $\mathcal{F}$)} when $\mathcal{L}_V\sigma =0$ for every vertical field $V\in\Gamma(T\mathcal{F})$. Therefore, it becomes easy to check that \textit{the space of transverse vector fields coincides with the vector subspace of holonomy-invariant sections of the normal bundle $\nu \f$.}

Although the leaf space $M/\f$ can be very nasty from a differential-topological standpoint and will have a smooth manifold structure in very few cases, one often looks at ``transverse geometric/analytic objects'' --- a somewhat vague phrase better illustrated via concrete examples --- as the ``non-singular'' version of the corresponding objects on the leaf space. Thus, for example, a basic smooth function $f\in C^\infty(\f)$ can be viewed as a ``smooth function'' on $M/\f$, transverse vector fields as ``smooth vector fields'' on $M/\f$ and so on. 

The next proposition illustrates how this philosophy works for (globally defined) smooth submersions and will be concretely used later on the paper. Although very simple and probably quite well-known, we present a proof it here for completeness, since its parts are somewhat scattered in the extant literature. Recall that if $F:M\rightarrow N$ is a smooth map (not necessarily a submersion) then vector fields $X\in \mathfrak{X}(M)$ and 
$Y\in \mathfrak{X}(N)$ are \textit{$F$-related} if 
$$dF_x(X_x) = Y_{F(x)}, \quad \forall x\in M,$$
a fact we denote as $X\stackrel{F}{\sim} Y$. Lie brackets are preserved by $F$-relatedness in the sense that for any $X,X'\in \mathfrak{X}(M)$ and any $Y, Y'\in \mathfrak{X}(N)$, if $X\stackrel{F}{\sim} Y$ and $X'\stackrel{F}{\sim} Y'$, then $[X,X']_M \stackrel{F}{\sim} [Y,Y']_N$. Here and hereafter, when $F:M\rightarrow N$ is a globally defined onto submersion, any mention to a foliation will always refer to that foliation $\f$ whose leaves are the connected components of the fibers of $F$. Note that in this context, a vector field $V\in \mathfrak{X}(M)$ is vertical if and only if it is $F$-related to the zero vector field on $N$. 
\begin{lemma}\label{beconcrete}
    The following statements hold for an onto submersion $F:M\rightarrow N$. 
    \begin{itemize}
        \item[i)] Given $X\in \mathfrak{X}(N)$, there exists a $X^\ast \in \mathfrak{X}(M)$ such that $X^\ast \stackrel{F}{\sim} X$, called a \emph{lift} or \emph{pullback} of $X$ (through $F$). 
        \item[ii)] Any lift of $X\in \mathfrak{X}(N)$ is projectable, and if $X^\ast, \hat{X}^\ast$ are two lifts of the same $X$, then $X^\ast - \hat{X}^\ast \in \mathfrak{X}(\f)$. i.e., they differ by a vertical vector field and hence have the same projection as transverse vector fields of $\f$. 
        \item[iii)] The map 
        \begin{equation}\label{eqbe}
            \phi: X\in \mathfrak{X}(N) \mapsto \overline{X^\ast} \in \mathfrak{l}(\f)
        \end{equation}
        is well-defined, and it is a Lie algebra monomorphism, i.e., it is linear, injective and preserves Lie brackets. 
        \item[iv)] If all the fibers of $F$ are connected, then the map $\phi$ in the previous item is also surjective, and hence an isomorphism of Lie algebras. 
    \end{itemize}
\end{lemma}
\begin{proof} $(i)$\\
Fix a horizontal distribution $\mathcal{H}\subset TM$. Given $x\in M$, since $dF_x : T_xM \rightarrow T_{F(x)}N$ is onto and $T\f_x = \ker dF_x$, $dF_x$ maps $\mathcal{H}_x$ isomorphically onto $T_{F(x)}N$; hence, given $X\in \mathfrak{X}(N)$ we may define $X^\ast_x \in \mathcal{H}_x$ as the unique vector such that $dF_x(X^\ast_x)= X_x$. Doing that pointwise, we define a smooth $X^\ast \in \Gamma(\mathcal{H})$ which is a lift of $X$ by construction.\\
$(ii)$\\
Given any lift $X^\ast \in \mathfrak{X}(M)$ of $X\in \mathfrak{X}(N)$ and a vertical $V\in \mathfrak{X}(\f)$, we have 
$$X^\ast \stackrel{F}{\sim} X \, , \, V \stackrel{F}{\sim} 0_{\mathfrak{X}(N)} \Rightarrow [X^\ast,V]_M \stackrel{F}{\sim} [X,0_{\mathfrak{X}(N)}]_N\equiv 0_{\mathfrak{X}(N)},$$
 or in other words, $[X^\ast, V]_M$ is vertical, and hence $X^\ast \in \mathfrak{L}(\f)$ as desired. Given another lift $\hat{X}^\ast$, we have
 $$dF_x(X^\ast_x)= X_{F(x)}= dF_x(\hat{X}^\ast_x) \Rightarrow dF_x((X^\ast-\hat{X}^\ast)_x) =0, \quad \forall x\in M,$$
 whence the verticality of $X^\ast - \hat{X}^\ast$ follows. \\
 $(iii)$\\
 That $\phi$ is well-defined follows immediately from $(ii)$, and the proof of its linearity is also straightforward. $X\in \ker \phi$ means it admits a vertical lift, which is therefore $F$-related to the zero field on $N$ and to $X$, which entails that $X\equiv 0_{\mathfrak{X}(N)}$, that is, $\phi$ is one-to-one. Now, let $X,Y \in \mathfrak{X}(N)$ and consider lifts $X^\ast,Y^\ast \in \mathfrak{X}(M)$ of $X$ and $Y$, respectively. Thus we have $[X^\ast,Y^\ast]_M \stackrel{F}{\sim} [X,Y]_N$, that is, $[X^\ast,Y^\ast]_M $ is a lift of $[X,Y]_N$ and hence
 $$\phi([X,Y]_N) = \overline{[X^\ast,Y^\ast]_N} = [\overline{X^\ast},\overline{Y^\ast}]_{\mathfrak{l}(\f)} \equiv [\phi(X),\phi(Y)]_{\mathfrak{l}(\f)}.$$
 \\
 $(iv)$\\
 Assume that all fibers of $F$ are connected, and fix $Z\in \mathfrak{L}(\f)$. We define $X\in \mathfrak{X}(N)$ as follows. Given $x\in N$, pick any $z\in F^{-1}(x)$, and put $X_x:= dF_z(Z_z)$. To see this is well-defined pick $f\in C^\infty(N)$, and define the smooth real-valued function $\xi \in C^\infty(M)$ given by
 $$\xi(z) := Z_z(f\circ F) = dF_z(Z_z), \quad \forall z\in M.$$
 We claim that $\xi$ is basic: since $f$ was taken arbitrarily this does imply that $X$ is well-defined since $F^{-1}(x)$, being connected, coincides with a leaf of $\f$. Thus, pick $V \in \mathfrak{X}(\f)$. We have 
 $$V\xi = [V,Z]_M(f\circ F) + ZV(f\circ F).$$
 But since $[V,Z]_M$ is vertical,
 $$([V,Z]_M)_z(f\circ F) = dF_z(([V,Z]_M)_x)\equiv 0,$$
 and similarly we have $V(f\circ F) \equiv 0$. whence we conclude that $V\xi\equiv 0$, and the claim follows, thus finishing the proof. 
\end{proof}
\begin{remark}
    With the conditions and notations of the previous Lemma, some added comments are in order. 
    \begin{itemize}
        \item[a)] Projectable vector fields carry a natural $C^\infty(N)$-module structure defined by 
        $$(f, X) \in C^\infty(N)\times \mathfrak{L}(\f) \mapsto (f\circ F)\cdot X \in \mathfrak{L}(\f)$$
        which induces a $C^\infty(N)$-module structure on the transverse vector fields as well. Moreover, given any $X \in \mathfrak{X}(N)$ and $f\in C^\infty(N)$, and any lift $X^\ast \in \mathfrak{L}(\f)$ of $X$, $(f\circ F)\cdot X^\ast$ is a lift of $f\cdot X$, so $\phi$ defined in Lemma \ref{beconcrete}(ii) is a also a monomorphism of $C^\infty(N)$-modules, and in the conditions of $(iii)$, a module isomorphism as well. This strongly vindicates the identification of transverse fields with ``(desingularized) vector fields on the leaf space''. 
        \item[b)] In his celebrated paper \cite{oneil} on Riemannian submersions, O'Neill refers to a vector field on $M$ as \textit{basic} if it is the (horizontal) lift of a vector field on $N$. He also points out that not all horizontal vector fields are basic, and that basic fields are in one to one-correspondence with vector fields on $N$. Similarly, we see in Lemma \ref{beconcrete}(i) that all basic fields are projectable, but not all projectable fields are lifts of vector fields on $N$ if the fibers of $F$ are not connected. For example, if $M$ is $\mathbb{R}^2$ with the non-negative $y$-axis deleted, consider the submersion $\pi:(x,y) \in  M \mapsto y\in \mathbb{R}$. Fix any smooth function $\varphi \in C^\infty(\mathbb{R})$ such that $\text{supp}\, \varphi \subset (0,+\infty) $ and $\varphi|_{[1,+\infty)}=1$, and let $X\in \mathfrak{X}(M)$ be given as
        $$X(x,y) = \begin{cases}
            \varphi(y)\cdot \partial/\partial y, \text{ if $x>0$},\\
            - \varphi(y)\cdot \partial/\partial y, \text{ if $x<0$},\\
            0, \text{ if $y<0$.}
        \end{cases}$$
        Then, $X$ is projectable, but not a lift of a vector field on $\mathbb{R}$. 
    \end{itemize}
\end{remark}
A $(0,s)$-tensor $T$ on $M$ is said to be \textit{holonomy-invariant (with respect to $\mathcal{F}$)} when $\mathcal{L}_VT=0$. A very important special case is the following: a symmetric $(0,2)-$tensor $g_\intercal$ is called a \textit{semi-Riemannian transverse metric} on $(M,\f)$ if $(i) $ it is holonomy invariant, $(ii)\;  g_\intercal(X,V)=0,\forall X\in\mathfrak{X}(M)\iff V\in\mathfrak{X}(\mathcal{F})$, and $(iii)$ for every $x\in M$ the index of the bilinear form $g_\intercal(x):T_xM\times T_xM\rightarrow\mathbb{R}$ is constant number $0\leq \nu\leq q$. This constant number is a called the \textit{index} of $g_\intercal$. When $\nu=0$, $(M,\f,g_\intercal)$ is called a \textit{Riemannian foliation}. Another case of interest is when $\nu=1$ and $q\geq 2$; we then say that $g_\intercal$ is a \textit{Lorentzian transverse metric} and $(M,\f,g_\intercal)$ is a \textit{Lorentzian foliation}. 

For the rest of the paper, whenever we do computations in local coordinates we shall adopt the Einstein's summation convention: repeated indices indicate summation. However, we slightly modify it as follows. Letters in the beginning of the Latin alphabet \textit{a, b, c, \ldots} denote summation from 1 through $q$, often corresponding to local coordinate vector fields in transversal directions, whereas letters in the middle of the Latin alphabet \textit{i, j, k,\ldots} denote summation over $q+1,\ldots, n$ and will often indicate coordinate vector fields along vertical directions. The use of Greek letters as repeated indices indicate a full summation from 1 through $n$.

\section{Transverse affine geometry}\label{sec3}

What we understand here as an \textit{affine manifold} is simply a pair $(M,\nabla)$ where $\nabla$ is an affine connection on $M$, i.e., a linear connection on $TM$. In particular every semi-Riemannian manifold is an affine manifold when equipped with the associated Levi-Civita connection of the metric. Our goal here is to discuss a transverse analogue of this notion for foliations, thus defining the notion of an \textit{transversely affine foliation}. We fix throughout a foliated manifold $(M,\f)$. 

Now, attempts to define an appropriate notion of a transversely affine foliation are not new in the foliation literature, so we briefly review some of these here. In \cite{furness-fedida}, for instance, the authors posit that $\mathcal{F}$ is a \textit{transversely affine foliation} when the transition diffeomorphisms of a Haefliger cocycle defining $\mathcal{F}$ are affine functions of $\mathbb{R}^q$. That is, for each $\alpha,\beta\in\Lambda$ such that $\mathcal{U}_\alpha\cap\mathcal{U}_\beta\neq\emptyset$, there are $A_{\alpha\beta}\in GL(q,\mathbb{R)}$ and $\textbf{b}_{\alpha\beta}\in\mathbb{R}^q$ such that $\gamma_{\alpha\beta}(\textbf{x})=A_{\alpha\beta}\text{x}+\textbf{b}_{\alpha\beta}, \forall\textbf{x}\in s_\alpha(\mathcal{U}_\alpha\cap\mathcal{U}_\beta)$. However, this definition is too restrictive: it is implicitly only dealing with the \textit{(canonical) flat connection} on $\mathbb{R}^q$. A more general notion of a transverse affine structure can be found in Wolak's work \cite{wolak2}, albeit with another name and merged with additional structures. The idea is that the maps defining a Haefliger cocycle $\gamma_{\alpha\beta}$ are all (general) \textit{affine transformations} between open sets of $\mathbb{R}^q$ in the following more general sense. For each $\alpha\in\Lambda$ one assumes there are defined affine connections $\underaccent{\check}{\nabla}^{\alpha}$ on $\underaccent{\check}{\mathcal{U}}_\alpha$ such that whenever $\gamma_{\alpha\beta}$ exists, it satisfies $\gamma^*_{\alpha\beta}\underaccent{\check}{\nabla}^{\beta}=\underaccent{\check}{\nabla}^{\alpha}$, where $\nabla^\ast$ denote the pullback connection\footnote{Recall that if $F:M\rightarrow N$ is a diffeomorphism and $\nabla$ is an affine connection on $N$, the \textit{pullback} of $\nabla$ is the affine connection $\nabla^\ast$ on $M$ given by the formula $$F_\ast(\nabla^\ast _XY) = \nabla_{F_\ast X}F_\ast Y, \quad \forall X,Y \in \mathfrak{X}(M).$$ }. The latter definition indeed accords well with Haefliger's well-known pseudogroup approach to transverse geometry of foliations \cite{haefliger3,haefliger4,mrcun}, and we shall essentially adopt it in this work. However, it has the drawback that a proof of the independence of the choice of a Haefliger's cocycle is not obvious. Accordingly, we shall seek an equivalent geometric formulation of Wolak's definition which is explicitly independent of that choice. 
\begin{remark}\label{basicconn}
    There is an invariant, natural geometric context behind Wolak's approach that can be traced back to the theory of \textit{basic/projectable connections on transverse frame bundles} developed by Molino \cite{molino}, Bott \cite{bott}, Kamber and Tondeur \cite{kambertondeur}. Again the interest of making this discussion relatively self-contained, we briefly recall its main points here.

Let $F_\intercal(M,\mathcal{F})$ denote the $GL(q,\mathbb{R})$-principal bundle of all frames on $\nu\mathcal{F}$. The foliation $\mathcal{F}$ on $M$ induces a foliation $\tilde{\mathcal{F}}$ of the same dimension on $F_\intercal(M,\mathcal{F})$, called the \textit{lifted foliation} (see  \cite[Example 4.19]{mrcun}). A principal connection $\omega:\mathfrak{X}(F_\intercal(M,\mathcal{F}))\rightarrow\mathfrak{gl}(q,\mathbb{R})$ on this bundle is said to be \textit{transverse} (see \cite[Sec. 2.5]{molino}) if $\omega(X)=0$ for every vector field $X \in \mathfrak{X}(F_\intercal(M,\mathcal{F}))$ which is tangent to $\tilde{\mathcal{F}}$. That is to say, $\omega$ is transverse when it annihilates any $\tilde{\mathcal{F}}$-vertical vector. This is also what Kamber and Tondeur call an \textit{adapted connection}. These authors show \cite[Lemma 2.7]{kambertondeur} that such an object corresponds to a linear connection $D$ on $\nu\mathcal{F}$ . Furthermore, Molino, in Ref. \cite[Sec. 2.5]{molino} says that $\omega$ is a \textit{projectable connection} if it is a basic 1-form with respect to the lifted foliation $\tilde{\mathcal{F}}$, that is, it is both transverse/adapted and holonomy-invariant. In the work of Kamber and Tondeur, this corresponds to the concept of a \textit{basic connection}, and these authors show that such an object exists precisely when the foliation is transversely affine in Wolak's sense \cite{wolak2}.
\end{remark}
\medskip

Again, a disadvantage of the approach via basic connections on the principal bundle of transverse frames is that it gives a somewhat opaque framework to treat geodesics. Accordingly, we seek another alternative, and show it gives the same information as basic connections. 

A natural starting point would be to consider linear connections $D: \mathfrak{X}(M)\times \Gamma(\nu \f) \rightarrow \Gamma(\nu \f)$ on the normal bundle of $\f$. Recall the Lie derivative along a \textit{projectable} vector field $X\in \mathfrak{L}(\f) $ was defined in the previous section (conf. \eqref{lieder}). We now define
\begin{equation}\label{lieconnformula}
(\mathcal{L}_XD)_Z\sigma := \mathcal{L}_X(D_Z\sigma) - D_{[X,Z]}\sigma - D_Z(\mathcal{L}_X\sigma), \quad \forall Z\in \mathfrak{X}(M).
\end{equation}
It is easy to see $\mathcal{L}_XD$ is $C^\infty(M)$-bilinear in its entries $Z,\sigma$. We might then try and say that $D$ is holonomy-invariant if $\mathcal{L}_VD$ is identically zero for any vertical vector field $V$. It turns out, however, that this definition would be unnecessarily restrictive, because it would affect \textit{all} sections of the normal bundle. A better definition arises from the following simple fact whose proof is immediate from \eqref{lieconnformula} if we recall that 
$$\sigma \in \mathfrak{l}(\f) \Leftrightarrow \mathcal{L}_V\sigma =0, \quad \forall V\in \Gamma(T\f).$$
\begin{proposition}\label{definetransverseaffine}
    Let $D$ denote a connection on the normal bundle $\nu \f$ of the foliated manifold $(M,\f)$. The following statements are equivalent. 
    \begin{itemize}
        \item[i)] For any $X \in \mathfrak{L}(\f)$, $D_X\mathfrak{l}(\f)\subset \mathfrak{l}(\f)$. (In words: covariant derivatives of transverse vector fields along projectable vector fields are also transverse vector fields.) 
        \item[ii)] $(\mathcal{L}_VD)_X\sigma =0 \quad \forall V\in \mathfrak{X}(\f),\forall X\in \mathfrak{L}(\f), \forall \sigma \in \mathfrak{l}(\f).$
    \end{itemize}
\end{proposition}
\qcd
The previous result provides a more appropriate notion of holonomy-invariance for connections on the normal bundle $\nu \f$.
\begin{definition}\label{transaffdefi}
  A  connection on the normal bundle $\nu \f$ of the foliated manifold $(M,\f)$ is \emph{holonomy-invariant} if one (and hence both) of the properties in Prop. \ref{definetransverseaffine} holds. 
\end{definition}

\begin{example}[Basic Levi-Civita connections]
    The normal bundle $\nu\f$ is canonically equipped with a partial connection $\nabla^{\mathrm{Bott}}:\mathfrak{X}(\f)\times\Gamma(\nu\f)\rightarrow\Gamma(\nu\f)$, the so-called \textit{Bott connection}. This is given by $\nabla^{\mathrm{Bott}}_V\overline{X}:=\overline{[V,X]}, \forall V\in\mathfrak{X}(\f), \forall X\in\mathfrak{X}(M)$. If $(M,\f,g_\intercal)$ is a semi-Riemannian foliation and we pick a bundle-like metric $g$ with Levi-Civita connection $\nabla$, then the Bott connection can be completed to a full connection on $\nu\f$, called the \textit{basic (or transverse) Levi-Civita connection} $\nabla^\intercal:\mathfrak{X}(M)\times\Gamma(\nu\f)\rightarrow\Gamma(\nu\f)$ given by:$$
\nabla^\intercal_X\overline{Y}:=\begin{cases}
\nabla^{\mathrm{Bott}}_X\overline{Y}, \text{ if }X\in\mathfrak{X}(\f);\\
\overline{\nabla_XY}, \text{ if }X\in\mathfrak{X}(\f)^\perp.
\end{cases}
$$
One can easily check that the basic Levi-Civita connection does not depend on the specific bundle-like metric used. A long but straightforward computation shows that it is indeed holonomy-invariant in the sense of Definition \ref{transaffdefi}. 
\end{example}
\medskip

An important caveat of the previous definitions for us in this paper is that we are crucially interested in studying a transverse analogue of \textit{geodesics}, and it turns out that these are best understood in terms of (standard) geodesics of an affine connections \textit{on the manifold} $M$, \textit{not} on the normal bundle. We therefore need to introduce here the affine analogue of a bundle-like metric. 

\begin{definition}[Bundle-like affine structure]\label{bundle-like affine structure}
A \textit{bundle-like affine structure} on $(M, \mathcal{F})$ is a pair $(\mathcal{H}, \nabla)$, where $\nabla$ is an affine connection on $M$ and $\mathcal{H}\subset TM$ is a horizontal distribution such that $\nabla_{\mathcal{H}X}\mathcal{H}Y$ is projectable for every $X,Y\in\mathfrak{L}(\mathcal{F})$.\end{definition}

It is clear that independently of the choice of the horizontal distribution $\mathcal{H}$, the horizontal part of any projectable vector field is also projectable. Hence, given a bundle-like affine structure $(\mathcal{H}, \nabla)$ on $(M, \mathcal{F})$ we induce a connection $\overline{\nabla}$ on the normal bundle $\nu \f$ by 
$$\overline{\nabla }_X\sigma:= \overline{\nabla_{X}\mathcal{H}Z}, \quad \forall \sigma \in \Gamma(\nu \f), \forall X\in \mathfrak{X}(M),$$
where $Z \in \mathfrak{X}(M)$ is any vector field such that $\sigma = \overline{Z}$. We then easily see that $\overline{\nabla}_X \mathfrak{l}(\f) \subset \mathfrak{l}(\f), \forall X\in \mathfrak{L}(\f)$, and hence $\overline{\nabla}$ is indeed holonomy-invariant in the sense of Definition \ref{transaffdefi}. 

Recall that if $(M, \mathcal{F})$ is endowed with a bundle-like \textit{semi-Riemannian metric} $g$, then it is locally given by semi-Riemannian submersions. As Abe and Hasegawa point out (more on that in Section \ref{section: geodesics}), one can then consider affine submersions with horizontal distribution $\mathcal{H}$ given by the orthogonal complement with respect to the metric. The decomposition $TM= \mathcal{H}\oplus T\f$ is well-defined since the leaves of $\mathcal{F}$ are semi-Riemannian submanifolds (see, for instance, chapter 4 of \cite{oneillbook}). This means that a semi-Riemannian foliation will automatically define a bundle-like affine structure with $\nabla$ chosen as the Levi-Civita connection. 

\begin{comment}
Another important insight on the subject of affine foliations arise in Lewis' work \cite{lewis}, in which the general relation between connections and distributions (not necessarily integrable) is studied. Following Lewis, we say a smooth distribution $\mathcal{D}$ on an affine manifold $(M,\nabla)$ is \textit{geodesically invariant} if for every $\nabla$-geodesic $\gamma:[a,b]\rightarrow M$ with $\dot\gamma(a)\in\mathcal{D}_{\gamma(a)}$ we have $\dot\gamma(t)\in\mathcal{D}_{\gamma(t)}, \forall t\in(a,b)$. It is clear that being geodesically invariant in this sense is the same as saying that the \textit{geodesic spray} of $\nabla$, which we denote by $\mathcal{G}\in\mathfrak{X}(TM)$, is everywhere tangent to $\mathcal{D}$ when latter is seen as a submanifold of $TM$. Lewis goes on what he calls the \textit{symmetric product}: an $\mathbb{R}$-bilinear map $\langle X:Y\rangle:=\nabla_XY+\nabla_YX$, which, regarding geodesic invariance, plays an analogous role in geodesic invariance as Lie brackets do in integrability. More precisely, it yields the following equivalences \cite[Thm. 4.4]{lewis}.

\begin{theorem}\label{from-lewis}
The following are equivalent:
\begin{enumerate}
    \item $\mathcal{D}$ is geodesically invariant;
    \item $X,Y\in\Gamma(\mathcal{D})\implies \langle X:Y\rangle\in\Gamma(\mathcal{D})$;
    \item $X\in\Gamma(\mathcal{D})\implies \nabla_XX\in\Gamma(\mathcal{D})$.
\end{enumerate}
\end{theorem}
\qcd
\end{comment}
One mildly annoying feature of bundle-like affine structures is their dependence on the specific choice of the horizontal distribution $\mathcal{H}$. It would be interesting to seek out a definition of affine structure which is ``purely transverse'' in the sense of not depending on such a choice, because it is --- unlike with semi-Riemannian foliations --- somewhat \textit{ad hoc} and highly non-unique. We thus propose the following definition.

\begin{definition}[Transverse affine connection]\label{tasdefi1}
Let $\Hat{\nabla}$ be an affine connection on $M$. We say that $\Hat{\nabla}$ is a \textit{transverse affine connection} (on $(M,\mathcal{F})$) if there is a torsion-free linear connection $\omega:\mathfrak{X}(M)\times\Gamma(T\mathcal{F})\rightarrow\Gamma(T\mathcal{F})$, which we refer to as a \textit{partner connection}, such that, for all $X\in\mathfrak{X}(M)$ and all $V\in\Gamma(T\mathcal{F})$ we have the properties:
\begin{enumerate}
    \item $\Hat{\nabla}_XV=\omega_XV$;
    \item $\Hat{\nabla}_VX=[V,X]+\omega_XV$;
    \item $\mathcal{L}_V\Hat{\nabla}$ is a $(1,2)$-tensor taking values on $\Gamma(T\mathcal{F})$.
\end{enumerate}
\end{definition}

Note that the linear connection $\omega$ in the previous definition is required to be torsion-free in order to make conditions $(1)$ and $(2)$ compatible when $\Hat{\nabla}$ is evaluated on two vertical vector fields. That is, if $V, W\in\Gamma(T\mathcal{F})$ we have that \begin{align*}
\Hat{\nabla}_VW\overset{(1)}{=}\omega_VW\overset{\mathrm{torsion\mbox{-}free}}{=}[V, W]+\omega_WV\overset{(2)}{=}\Hat{\nabla}_VW.
\end{align*}
In particular, a given transverse affine connections induce a linear connection on the bundle $T\f$, to wit, its partner connection.

Condition (3) is meant to ensure that $\Hat{\nabla}$ again induces an holonomy-invariant connection on the normal bundle. Indeed, if $X, Y\in \mathfrak{L}(\mathcal{F})$ and $V\in\mathfrak{X}(\mathcal{F})$, then \begin{align}\label{derivative is projectable}
\left[V, \hat{\nabla}_XY\right]&=\left(\mathcal{L}_V\Hat{\nabla}\right)(X,Y)+\hat{\nabla}_{[V,X]}Y+\hat{\nabla}_X[V,Y]\\
&=\left(\mathcal{L}_V\Hat{\nabla}\right)(X,Y)+[[V,X],Y]+\omega_{Y}[V,X]+\omega_X[V,Y],
\end{align}which is everywhere vertical. Therefore, the covariant derivative of projectable vector fields with respect to a transverse affine connection is again a projectable field. In particular, if we choose \textit{any} horizontal distribution $\mathcal{H}$, $\mathcal{H}X$ is projectable whenever $X$ is projectable, and therefore we have:
\begin{proposition}\label{prop: tac produces bas}
    If $\Hat{\nabla}$ is a transverse affine connection on $(M,\f)$ and $\mathcal{H}$ is a horizontal distribution, then $(\Hat{\nabla}, \mathcal{H})$ is a bundle-like affine structure on $(M,\f)$.
\end{proposition}
\qcd
\begin{example}
\textit{The converse of Prop. \ref{prop: tac produces bas} is false in general.} For example, consider the 2-dimensional foliation $\mathcal{F}$ given by the submersion $F:\mathbb{R}^3\rightarrow\mathbb{R}$, $F(x,y,z)=z-x^2-y^2$. The Euclidean metric $g=dx^2+dy^2+dz^2$ is \textit{not} a bundle-like metric with respect to $\mathcal{F}$, but its Levi-Civita connection, which is the canonical flat connection $\nabla^\mathrm{flat}$ produces a bundle-like affine structure together with $\mathcal{H}:=T\mathcal{F}^\perp$. To see that, we start by noting that \begin{align*}
    T\mathcal{F}&=\mathrm{span}\left\{V_1:=\frac{\partial}{\partial x}+2x\frac{\partial}{\partial z}, V_2:=\frac{\partial}{\partial y}+2y\frac{\partial}{\partial z}\right\}, \\
    \mathcal{H}&=\mathrm{span}\left\{H:=\frac{\partial}{\partial z}-2x\frac{\partial}{\partial x}-2y\frac{\partial}{\partial y}\right\}.
\end{align*}A simple computation shows that $H$ is projectable, and $\nabla^\mathrm{flat}_HH=-4\left(x\frac{\partial}{\partial x}+y\frac{\partial}{\partial y}\right)$. And then,
\begin{align*}
[V_1,\nabla^\mathrm{flat}_HH]&=\left[\frac{\partial}{\partial x}+2x\frac{\partial}{\partial z}, -4\left(x\frac{\partial}{\partial x}+y\frac{\partial}{\partial y}\right)\right]=-4\frac{\partial}{\partial x}+8x\frac{\partial}{\partial z}=-4V_1\\
[V_2,\nabla^\mathrm{flat}_HH]&=\left[\frac{\partial}{\partial y}+2y\frac{\partial}{\partial z}, -4\left(x\frac{\partial}{\partial x}+y\frac{\partial}{\partial y}\right)\right]=-4\frac{\partial}{\partial y}+8y\frac{\partial}{\partial z}=-4V_2,
\end{align*}which means that $\nabla^{\mathrm{flat}}_HH$ is projectable. However, we may compute, for example, $$
\nabla^\mathrm{flat}_{\frac{\partial}{\partial x}}V_1=2\frac{\partial}{\partial z},
$$which is not vertical. Therefore, $\nabla^\mathrm{flat}$ is not an transverse affine connection. However, Proposition \ref{good-new-switcheroo} will show later on how to produce a transverse affine connection, given a bundle-like affine structure.
\end{example}

Observe that adding any $(1,2)$-tensor which takes values on $\Gamma(T\mathcal{F})$ to a transverse affine connection will produce another transverse affine connection with the same transverse information, in the following sense: if $\Hat{\nabla}^1$ and $\Hat{\nabla}^2$ are two transverse affine connections which differ by a $(1,2)$-tensor field taking values on the vertical bundle, then it is clear that they must induce the same connection on $\nu\mathcal{F}$. This motivates the following definition.

\begin{definition}[Transverse affine structure]\label{tasdefi2}
    A \textit{transverse affine structure} on $(M, \mathcal{F})$ is an equivalence class $[\Hat{\nabla}]$ of transverse affine connections with respect to the relation that identifies two such connections whenever their respective induced connections on $\nu\mathcal{F}$ coincide.
\end{definition}

The next result shows that Defs. \ref{tasdefi1} and \ref{tasdefi2} work in tandem to capture the same information as holonomy-invariant connections on the normal bundle.
\begin{proposition}\label{good-new-switcheroo}
Let $(M,\mathcal{F})$ be a foliated manifold. The following statements hold.
\begin{itemize}
    \item[i)] If $(M,\mathcal{F})$ admits a bundle-like affine structure $(\mathcal{H}, \nabla)$, then there exists a unique transverse affine structure $[\Hat{\nabla}]$ such that $$\mathcal{H}\nabla_XY=\overline{\nabla}_XY, \forall X,Y\in\Gamma(\mathcal{H}),\forall \overline{\nabla}\in[\hat{\nabla}].$$
    \item[ii)] There is a one-to-one correspondence between the collection of transverse affine structures on $(M,\mathcal{F})$ and the set of holonomy-invariant connections on the normal bundle $\nu \f$. 
\end{itemize}
\end{proposition}
\begin{proof}

$(i)$\\
Given a bundle-like affine structure $(\mathcal{H}, \nabla)$, we define $\hat{\nabla}$ by the formula$$
\Hat{\nabla}_XY:=\mathcal{V}\left(\nabla_X\mathcal{V}Y\right)+[\mathcal{V}X, \mathcal{H}Y]+\mathcal{V}\left(\nabla_{\mathcal{H}Y}\mathcal{V}X\right)+\mathcal{H}\left(\nabla_{\mathcal{H}X}\mathcal{H}Y\right) \quad \forall X,Y \in \mathfrak{X}(M).
$$ A straightforward computation shows that $\hat{\nabla}$ is indeed a transverse affine connection with $\mathcal{V}\left(\nabla_{(\cdot)}\mathcal{V}(\cdot)\right)$ defining a partner connection for it. In addition, it is clear that $[\hat{\nabla}]$ satisfies the stated property as well.\\
$(ii)$\\
Denote by $\mathcal{TA}(\f)$ the set of transverse affine structures on $(M,\f)$ and by $\mathfrak{H}(\nu \f)$ the set of holonomy-invariant connections on $\nu\f$. It suffices to show that any holonomy-invariant connection on $\nu\f$ is induced by a unique transverse affine structure. Choose a horizontal distribution $\mathcal{H}$ and any affine connection $\nabla$ on $M$. Observe that the choice of $\mathcal{H}$ induces a $C^\infty(M)$-module isomorphism $\Psi:\Gamma(\mathcal{H})\rightarrow \Gamma(\nu \f) $ by $\Psi(X) = \overline{X}, \forall X\in \Gamma(\mathcal{H})$. Given $D \in \mathfrak{H}(\nu \f)$, define the connection on $M$ given by 
$$\nabla^D_XY := \mathcal{V}\nabla _X \mathcal{V}Y + \Psi^{-1}(D_X\Psi(\mathcal{H}Y)), \quad \forall X,Y \in \mathfrak{X}(M).$$
We claim that $(\mathcal{H},\nabla^D)$ is a bundle-like affine structure. If we can establish this, then by $(i)$ the proof will be complete. To see it, let $X,Y\in \mathfrak{L}(\f)$, so that $\mathcal{H}X,\mathcal{H}Y \in \mathfrak{L}(\f)$ as well, and we compute
$$\nabla^D_{\mathcal{H}X}\mathcal{H}Y = \Psi^{-1}(D_{\mathcal{H}X}\overline{\mathcal{H}Y}) \in \mathfrak{L}(\f)$$
because $\overline{\mathcal{H}Y}\in \mathfrak{l}(\f)$, and hence $D_{\mathcal{H}X}\overline{\mathcal{H}Y} \in \mathfrak{l}(\f)$ by holonomy-invariance. 
\end{proof}
\begin{remark}
  When a bundle-like affine structure $(\mathcal{H}, \overline{\nabla})$ and a transverse affine structure $[\Hat{\nabla}]$ are related as in Props. \ref{prop: tac produces bas} and \ref{good-new-switcheroo}, we say that they are \textit{associated} with each other. They have a duality analogous to that between a transverse semi-Riemannian metric on $(M,\f)$ and an associated bundle-like metric. Bundle-like affine structures, holonomy-invariant connections on $\nu \f$ and transverse affine structures convey the same transverse affine information on $(M,\f)$, and we can go back and forth among them according to convenience. 
\end{remark}
\begin{comment}
A transverse affine connection --- and hence a transverse affine structure structure --- on $(M,\mathcal{F})$ can be defined provided we are given a linear connection $\overline{\nabla}$ on $\nu\mathcal{F}$, as long as it is holonomy-invariant in the sense that 
$$X,Y\in\mathfrak{L}(\mathcal{F})\implies \overline{\nabla}_X\overline{Y}\in\mathfrak{l}(\mathcal{F}),$$
or equivalently, when each of its Christoffel symbols is a basic function. To see this, pick any torsion-free connection $\omega$ on $T\mathcal{F}$ and let $(\mathcal{U}_\alpha)_{\alpha}$ be a open cover of $M$. For some $\mathcal{U}_\alpha$ in the cover, we write the Christoffel symbols $\overline{\Gamma}^a_{b\lambda}$ for $\overline{\nabla}$ and $\omega^i_{j\lambda}$ for $\omega$, with respect to a foliated chart defined on $\mathcal{U}_\alpha$. Then, for each $\lambda,\mu,\nu =1,\ldots n$ we define $\hat{\Gamma}^\mu_{\nu\lambda}:\mathcal{U}_\alpha\rightarrow\mathbb{R}$ by \begin{align*}
\hat{\Gamma}^\mu_{\nu\lambda}:=\overline{\nabla}^b_{a\lambda}\delta^\mu_b\delta^a_\nu+\omega^i_{j\lambda}\delta^\mu_i\delta^j_\nu+\omega^i_{j\nu}\delta^\mu_i\delta^j_\lambda.
\end{align*}A simple computation shows that those symbols define globally a transverse affine connection $\hat{\nabla}$ with partner connection $\omega$ such that its projection to $\nu\mathcal{F}$ coincides with $\overline{\nabla}$.
\end{comment}

We can also guarantee that given a Haefliger cocycle we reproduce Wolak's definition. 
\begin{proposition}\label{locally-affine}
A regular foliation $(M, \mathcal{F})$ admits a transverse affine structure on $(M, \mathcal{F})$ if, and only if, it is transversely affine [in the sense of Wolak's definition].
\end{proposition}
\begin{proof}
$(\impliedby)$ Suppose $\Hat{\nabla}$ is a transverse affine connection on $(M, \mathcal{F})$. Let $(\mathcal{U}_\alpha, s_\alpha, \gamma_{\alpha\beta})_{\alpha, \beta\in A}$ be a Haefliger's cocycle defining $\mathcal{F}$ and choose $\alpha, \beta\in A$ such that $\mathcal{U}_\alpha\cap\mathcal{U}_\beta\neq\emptyset$. To check Wolak's requirements, must construct affine connections on $\underaccent{\check}{\mathcal{U}}_{\alpha}$ and $\underaccent{\check}{\mathcal{U}}_{\beta}$ such that $\gamma_{\alpha\beta}$ is an affine transformation. Let $((x^a)_{a=1}^q,(x^i)_{i=q+1}^n)$ be a foliated system of coordinates on $\mathcal{U}_\alpha\cap\mathcal{U}_\beta$. Let $\Hat{\Gamma}^\lambda_{\nu\mu}$ be the Christoffel symbols of $\Hat{\nabla}$ with respect to the $(x^\mu)_{\mu=1}^n$ coordinates. Note that the Lie derivative condition on $\Hat{\nabla}$ implies that the vectors $$
\left(\mathcal{L}_{\frac{\partial}{\partial x^i}}\Hat{\nabla}\right)\left(\frac{\partial}{\partial x^\mu}, \frac{\partial}{\partial x^\nu}\right)=\left[\frac{\partial}{\partial x^i},\Hat{\nabla}_{\frac{\partial}{\partial x^\mu}}\frac{\partial}{\partial x^\nu}\right]=\left[\frac{\partial}{\partial x^i},\Hat{\Gamma}^\lambda_{\nu\mu}\frac{\partial}{\partial x^\lambda}\right]=\frac{\partial}{\partial x^i}\left(\Hat{\Gamma}^\lambda_{\nu\mu}\right)\frac{\partial}{\partial x^\lambda}
$$ are vertical. But this means precisely that $\frac{\partial \Hat{\Gamma}^a_{\mu\nu}}{\partial x^i}=0$, for $\mu,\nu=1,\ldots, n; a=1,\ldots, q; i=q+1,\ldots, n$. Thus, we are able to unambiguously define, for each $a, b, c=1,\ldots q$, functions ${}^\alpha\Gamma^c_{ba}$ and ${}^\beta\Gamma^c_{ba}$ respectively on $\underaccent{\check}{\mathcal{U}}_{\alpha}$ and $\underaccent{\check}{\mathcal{U}}_{\beta}$, by the relation $$
{}^\alpha\Gamma^c_{ba}\circ s_\alpha=\Hat{\Gamma}^c_{ba}={}^\beta\Gamma^c_{ba}\circ s_\beta.
$$
Now, by a simple ``project, modify and lift'' construction, we can assume that each $x^a$ has been chosen so that $\{\frac{\partial}{\partial x^a}:a=1,\ldots, q\}$ projects simultaneously to bases for $\mathfrak{X}(\underaccent{\check}{\mathcal{U}}_{\alpha})$ and $\mathfrak{X}(\underaccent{\check}{\mathcal{U}}_{\beta})$, which we denote respectively by $(E_a)_{i=1}^q$ and $(F_a)_{i=1}^q$. For any $X, Y\in \mathfrak{X}(\underaccent{\check}{\mathcal{U}}_{\alpha})$ write $X=X^aE_a, Y=Y^aE_a$ and for any $Z,W\in \mathfrak{X}(\underaccent{\check}{\mathcal{U}}_{\beta})$ write $Z=Z^aF_a$ and $W=W^aF_a$. Thus, we can define \begin{align*}
\underaccent{\check}{\nabla}^{\alpha}_XY&:=X^a\left(E_a(Y^c)+{}^\alpha\Gamma^c_{ba}Y^b\right)E_c\\
\underaccent{\check}{\nabla}^{\beta}_ZW&:=Z^a\left(F_a(W^c)+{}^\beta\Gamma^c_{ba}W^b\right)F_c
\end{align*}
It is easy to check that these formulae define affine connections and from the cocycle condition $\gamma_{\beta\alpha}\circ s_\alpha|_{\mathcal{U}_\alpha\cap\mathcal{U}_\beta}=s_\beta|_{\mathcal{U}_\alpha\cap\mathcal{U}_\beta}$ we have that $\gamma_{\beta\alpha*}(E_a)=F_a$ for all $a$. It readily follows that $$
\gamma^*_{\beta\alpha}\underaccent{\check}{\nabla}^{\beta}=\underaccent{\check}{\nabla}^{\alpha}
$$
$(\implies)$ Cover $M$ by foliated charts $(\mathcal{U}_\alpha, s_\alpha)_{\alpha\in A}$. By hypothesis, for each $\alpha\in A$ there is an affine connection $\nabla^{*\alpha}$ on $\mathcal{U}_\alpha$ such that whenever $\mathcal{U}_\alpha\cap\mathcal{U}_\beta\neq\emptyset$ we have $\nabla^{*\alpha}=\gamma^*_{\beta\alpha}\nabla^{*\beta}$. Let $(E_a)_{a=1}^q$ be a basis for $\mathfrak{X}(\mathcal{U}^{*}_{\alpha})$ and write $\nabla^{*\alpha}_{E_b}E_a={}^\alpha\Gamma^c_{ab}E_c$. Choose any torsion-free linear connection $\omega:\mathfrak{X}(M)\times\Gamma(T\mathcal{F})\rightarrow\Gamma(T\mathcal{F})$ and denote its restriction to $\mathcal{U}_\alpha$ by $\omega^\alpha$. Choose a basis $(X_\mu)_{\mu=1}^n$ for $\mathfrak{X}(\mathcal{U}_{\alpha})$ such that $s_{\alpha*}(X_a)=E_a$ for $a=1,\ldots, q$ and each $X_i$ is vertical for $i=q+1,\ldots, n$.
We define the connection $\nabla^\alpha$ on $\mathcal{U}_\alpha$ by giving its values on the chosen basis and extending by linearity and product rule:$$
\nabla^\alpha_{X_\nu}X_\mu=\delta^i_\mu\left(\omega^\alpha_{X_\nu}X_i\right)+\delta^i_\nu\left([X_i,X_\mu]+\omega^\alpha_{X_\mu}X_i\right)+\delta^b_\nu\delta^a_\mu\left({}^\alpha\Gamma^c_{ab}\circ s_\alpha\right)X_c.
$$ If we do this for every $\alpha\in A$, then the transversely affine hypothesis implies that $\nabla^\alpha$ and $\nabla^\beta$ agree on $\mathcal{U}_\alpha\cap\mathcal{U}_\beta$. Therefore, we can define an affine connection $\Hat{\nabla}$ on $M$ with partner connection $\omega$. By construction, each $\nabla^\alpha$ is holonomy-invariant, which means that the condition $\langle\mathcal{L}_V\Hat{\nabla}$ is vertical $\rangle$ is satisfied and therefore $\Hat{\nabla}$ is a transverse affine connection on $M$.
\end{proof}
The following result is now immediate from Proposition \ref{locally-affine} and the discussion in Remark \ref{basicconn}.
\begin{corollary}
  There is a one-to-one correspondence between transverse affine structures on the foliated manifold $(M,\f)$ and basic principal connections on the principal frame bundle $F_\intercal(M,\mathcal{F})$ of transverse frames.
\end{corollary}
\qcd

For the rest of this section we perform some consistency checks for our definitions in two important contexts: submersions and semi-Riemannian foliations.
\begin{proposition}\label{push-pull-submersions}
Let $F:M\rightarrow N$ be an onto smooth submersion with connected fibers, and denote by $\mathcal{F}$ the foliation given by the fibers of $F$. Consider the expression:
\begin{equation}\label{what-must-be-upheld}
F_*(\hat{\nabla}_XY)=\nabla_{(F_*X)}(F_*Y), \forall X,Y\in\mathfrak{L}(\f).    
\end{equation}
\begin{enumerate}
    \item [(a)] Suppose that $(M,\mathcal{F})$ is given a transverse affine connection $\hat{\nabla}$. Then, there exists a unique affine connection $\nabla$ on $N$ such that equation \eqref{what-must-be-upheld} holds.
    \item [(b)] Suppose $N$ is given an affine connection $\nabla$. Then, there exists a unique transverse affine structure $[\nabla^\prime]$ on $(M,\f)$ such that for every $\hat{\nabla}\in[\nabla^\prime]$ equation \eqref{what-must-be-upheld} holds.
    \end{enumerate}
In situation $(a)$, we write $\nabla=:F_*\hat{\nabla}$, and call it the \emph{pushforward} of $\hat{\nabla}$, whereas in situation $(b)$, we write $[\nabla^\prime]=:F^*\nabla$, and call it the \emph{pullback} of $\nabla$.
\end{proposition}

\begin{proof}
\begin{enumerate}
\item [(a)]  Given $U,V\in\mathfrak{X}(N)$, by Lemma \ref{beconcrete}(i) we consider respectively lifts $U^*, V^*$ of $U$ and $V$, which by Lemma \ref{beconcrete}(ii) are projectable; hence by equation \eqref{derivative is projectable}, we have that $\hat{\nabla}_{U^*}V^*$ is a projectable vector field. The projection $\overline{\hat{\nabla}_{U^\ast}V^\ast }\in \mathfrak{l}(\f)$ does not depend on the choice of lifts of $U$ and $V$. Indeed, if $U^{**}$ and $V^{**}$ are other lifts of $U$ and $V$, respectively, then we would have $U^{**}=U^*+W_1$ and $V^{**}=V^*+W_2$, for some $W_1,W_2\in\mathfrak{X}(\f)$ (cf. Lemma \ref{beconcrete}(ii)). Then
\begin{align*}
\hat{\nabla}_{U^{**}}V^{**}&=\hat{\nabla}_{U^*}V^*+\hat{\nabla}_{U^*}W_2+\hat{\nabla}_{W_1}V^*+\hat{\nabla}_{W_1}W_2\\
&=\hat{\nabla}_{U^*}V^*+\omega_{U^*}W_2+[W_1,V^*]+\omega_{V^*}W_1+\omega_{W_1}W_2,
\end{align*}where $\omega$ is a partner connection, and hence $\overline{\hat{\nabla}_{U^*}V^*}=\overline{\hat{\nabla}_{U^{**}}V^{**}}$ as claimed. Therefore, because we are assuming $F$ has connected fibers, Lemma \ref{beconcrete}(iv) implies that this uniquely defined transverse field projects to a unique $\mathcal{Z}(U,V) \in \mathfrak{X}(N)$; if we set $\nabla_UV:= \mathcal{Z}(U,V)$, then $\nabla$ thus defined is easily checked to be an affine connection on $N$ satisfying \eqref{what-must-be-upheld}. Uniqueness is also clear from the construction. 

\item [(b)] We begin by fixing some horizontal distribution $\mathcal{H}$. Given any $U\in \mathfrak{X}(N)$, $U^\ast\in \Gamma(\mathcal{H})$ will denote its unique $\mathcal{H}$-horizontal lift, which again is projectable by Lemma \ref{beconcrete}(ii). 

Define a $\mathbb{R}$-bilinear\footnote{Observe that for any basic $f\in C^\infty(\f)$, there exists a unique $f_\ast \in C^\infty(N)$ such that $f= f_\ast \circ F$, and $(f_\ast\cdot  U)^\ast = f\cdot U^\ast$, for any $U\in \mathfrak{X}(N)$.} map on horizontal foliate vector fields as: 
$$\overline{\nabla}: (X,Y) \in (\mathfrak{L}(\f)\cap \Gamma(\mathcal{H}))^2 \mapsto \overline{\nabla}_XY := (\nabla _{F\ast X}F_\ast Y)^\ast \in \mathfrak{L}(\f)\cap \Gamma(\mathcal{H}),$$
where we denote by $F_\ast X, F_\ast Y \in \mathfrak{X}(N)$ the unique vector fields on $N$ having $X,Y$ as horizontal lifts, respectively. We can show that for any basic function $f\in C^\infty(\f)$ we have 
\begin{equation}\label{eqparasub1}
\overline{\nabla}_{f\cdot X}Y = f\cdot \overline{\nabla}_XY \text{ and }\overline{\nabla}_X(f\cdot Y) = X(f)\cdot Y + f\cdot \overline{\nabla}_XY,
\end{equation}
so $\overline{\nabla}$ is a ``partially defined connection''. 

In order extend it to a full-blooded connection, we recall that horizontal vector fields can be locally generated by horizontal foliated fields in the following sense. Given any $x\in M$, we can pick a local frame of vector fields $U_1, \ldots, U_q$ defined on some open set $F(x)\in \mathcal{U}\subset N$, so $F|_{F^{-1}(\mathcal{U})}$ is a submersion onto $\mathcal{U}$ with connected fibers --- in fact the open set $F^{-1}(\mathcal{U})$ is \textit{saturated}, i.e. a union of the leaves of our $\f$. Hence the horizontal foliate lifts $U_1^\ast , \ldots, U_q^\ast$ form a local frame on $F^{-1}(\mathcal{U})$, and hence given any horizontal vector field $X \in \Gamma(\mathcal{H})$ (not necessarily projectable) there exist unique smooth functions $f_1,\ldots, f_q \in C^\infty(F^{-1}(U))$ such that 
$$X = \sum _{a=1}^q f_a \cdot U^\ast_a = f_aU^\ast_a, $$
where, recall, the last equality arises because of Einstein summation conventions we have adopted (conf. last paragraph of the previous section). 
Taking this fact into account, given any other horizontal vector field $Y \in \Gamma(\mathcal{H})$, if we write $Y = g_b\cdot U^\ast _b$ (again with our summation convention!), then we can define
\begin{equation}\label{eqparasub2}
    \overline{\nabla}^{\mathcal{U}}_XY:= f_aU^\ast_a(g_b)\cdot U^\ast_b + f_ag_b \cdot \overline{\nabla}_{U^\ast_a}U^\ast_b,
\end{equation}
this equation defines a smooth, locally defined horizontal vector field, which via straightforward computation can be shown not to depend on the choice of the local frame, and for any $f\in C^\infty(M)$, we easily see that equations analogous to \eqref{eqparasub1} hold here. We shall refer to this by saying that $\overline{\nabla}^{\mathcal{U}}$ is a (locally defined) ``connection-like'' map. By gluing pieces of the form \eqref{eqparasub2} on some open cover of $N$, these local pieces add up to a well-defined ``connection-like'' map 
$$\widetilde{\nabla}: (X,Y) \in \Gamma(\mathcal{H})^2 \mapsto \widetilde{\nabla}_XY \in \Gamma(\mathcal{H}).$$

Finally, pick an affine connection $\mathfrak{D}$ on $M$, and define the affine connection given by
$$
\Hat{\nabla}_XY:=\mathcal{V}\left(\mathfrak{D}_X\mathcal{V}Y\right)+[\mathcal{V}X, \mathcal{H}Y]+\mathcal{V}\left(\mathfrak{D}_{\mathcal{H}Y}\mathcal{V}X\right)+\mathfrak{\widetilde{\nabla}}_{\mathcal{H}X}\mathcal{H}Y\ \quad \forall X,Y \in \mathfrak{X}(M).
$$ 
Just as in the proof of Proposition \ref{good-new-switcheroo}(ii), $\Hat{\nabla}$ is a transverse affine connection with $\mathcal{V}\left(\nabla_{(\cdot)}\mathcal{V}(\cdot)\right)$ defining a partner connection for it. Thus, $[\Hat {\nabla}]$ is the desired transverse affine structure. 
\end{enumerate}
\end{proof}

\begin{remark}\label{push-pull-local-diffeo}
Another interesting situation is when we have two foliated manifolds $(M,\mathcal{F}^M)$ and $(N,\mathcal{F}^N)$, and a local diffeomorphism --- say, a covering map --- $F:M\rightarrow N$ which sends leaves of $\mathcal{F}^M$ into leaves of $\mathcal{F}^N$, and we want to move affine information between them. Note that $F_*$ maps $\mathcal{F}^M$-vertical vectors into $\mathcal{F}^N$-vertical vectors, but not necessarily surjectively. It will become clear that it is necessary that only $\mathcal{F}^M$-vertical vectors are mapped into $\mathcal{F}^N$-vertical vectors. For example, consider $id:\mathbb{R}^3\rightarrow\mathbb{R}^3$, where the first $\mathbb{R}^3$ is foliated by $T\mathcal{F}^1=\mathbb{R}\frac{\partial}{\partial y}$ and the second one is foliated by $T\mathcal{F}^2=\mathbb{R}\frac{\partial}{\partial x}\oplus\mathbb{R}\frac{\partial}{\partial y}$. It is clear that $id$ is a foliated diffeomorphism. The usual flat connection $\nabla^\mathrm{flat}$ of $\mathbb{R}^3$ is a transverse affine connection to both foliations, albeit associated with different partner connections $\omega^1$ and $\omega^2$ (each of them is just $\nabla^\mathrm{flat}$ restricted to the respective vertical distributions). However, $id^*\nabla^\mathrm{flat}=\nabla^\mathrm{flat}$ and $id^*\omega^2=\omega^2$, which does not result in a transverse affine connection in the source foliation.

To achieve our goal, we must require furthermore that $\mathcal{F}^M$ and $\mathcal{F}^N$ are of the same dimension, so that vertical vectors on $TM$ are in one-to-one correspondence with the vertical vectors in $TN$. Indeed, if $(N,\mathcal{F}^N)$ is endowed with a transverse affine connection $\hat{\nabla}^N$, we can define the \textit{pullback} $F^*\hat{\nabla}^N$ as follows: for $X, Y \in\mathfrak{X}(M)$, we set $(F^*\hat{\nabla}^N)_XY=Z$, where $Z$ is given by: For $p\in M$, let $\mathcal{U}$ be any neighborhood of $p$ for which $\Phi:=F|_\mathcal{U}$ is a diffeomorphism onto its (open) image. Then, $Z(p):=d\Phi_p^{-1}\left(\hat{\nabla}^N_{\Phi_*X}\Phi_*Y\right)_{\Phi(p)}$. If $\omega^N$ is the partner connection associated with $\hat{\nabla}^N$, then we can check that, for any $V\in\mathfrak{X}(\mathcal{F}^M)$ and any $X\in\mathfrak{X}(M)$ we have\begin{align*}
(F^*\hat{\nabla}^N)_XV&=(F^*\omega^N)_XV,\\
(F^*\hat{\nabla}^N)_VX&=[V,X]+(F^*\omega^N)_XV.
\end{align*} As for the holonomy invariance condition, given $V\in\mathfrak{X}(\mathcal{F}^M)$, and $X, Y\in\mathfrak{L}(\mathcal{F}^M)$ we have \begin{align*}
(\mathcal{L}_V(F^*\hat{\nabla}^N))(X,Y)=[V, (F^*\hat{\nabla}^N)_XY]-(F^*\hat{\nabla}^N)_{[V,X]}Y-(F^*\hat{\nabla}^N)_X[V,Y].
\end{align*}
Note that the two properties that we already established imply that the two last terms above are vertical. As for the first one, note that since $F$ is locally a diffeomorphism, on a neighborhood of each point we can write $V=\Phi_*^{-1}(W)$ for some $W\in\mathfrak{X}(\mathcal{F}^N)$. Then we are left with $$
[V, (F^*\hat{\nabla}^N)_XY]=\left[\Phi_*^{-1}(W),\Phi_*^{-1}\left(\hat{\nabla}^N_{\Phi_*X}\Phi_*Y\right)\right]=\Phi_*^{-1}\left[W,\hat{\nabla}^N_{\Phi_*X}\Phi_*Y\right],
$$which is vertical because $\hat{\nabla}^N$ is a transverse affine connection and $\Phi_*$ stablishes a one-to-one correspondence between vertical vectors.

In the converse direction, if now $(M,\mathcal{F}^M)$ is given a transverse affine connection $\hat{\nabla}^M$, the necessary and sufficient condition for the existence of the pushforward via $F$ is that $F$ is one-to-one, (that is, a full diffeomorphism). When this is the case, we can set the \textit{pushforward} $(F_*\hat{\nabla}^M)_XY=W$, for $X,Y\in\mathfrak{X}(N)$, where $W$ is constructed as: given $q\in N$, let $\mathcal{U}$ be any neighborhood of $p=F^{-1}(q)$ for which $\Phi:=F|_\mathcal{U}$ is a diffeomorphism. Put $W(q):=d\Phi_p\left(\hat{\nabla}_{\Phi^*X}\Phi^*Y\right)_{p}$. As before, since vertical vectors in $TM$ and in $TN$ are in bijection, we obtain that $F_*\hat{\nabla}^M$ is a transverse affine connection with partner connection given by $F_*\omega^M$.
\end{remark}

The next important consistency check arises by considering semi-Riemannian foliations. We show that any one such gives rise to a unique transverse affine structure --- a result that works as an analogue of the fundamental theorem of semi-Riemannian geometry. 
\begin{theorem}\label{kind-of-a-fundamental-theorem}
Let $g_\intercal$ be a transverse semi-Riemannian metric on $(M, \mathcal{F})$. There exists a unique transverse affine structure $[\overline{\nabla}]$ on $(M, \mathcal{F})$ such that, for any $\hat{\nabla}\in[\overline{\nabla}]$ \begin{enumerate}
        \item [(i)] the torsion tensor $\mathrm{Tor}(\Hat{\nabla})$ takes values in $\Gamma(T\mathcal{F})$, and
        \item [(ii)] $Xg_\intercal(Y,Z)=g_\intercal(\Hat{\nabla}_XY,Z)+g_\intercal(Y,\Hat{\nabla}_XZ), \forall X,Y,Z\in\mathfrak{X}(M)$.
\end{enumerate}Moreover, any representative of $[\overline{\nabla}]$ projects down to the basic Levi-Civita connection of $g_\intercal$ on $\nu\mathcal{F}$.
\end{theorem}
\begin{proof}
We begin by noting that if $\Hat{\nabla}$ is a transverse affine connection satisfying conditions $(i)$ and $(ii)$, then it is straightforward to check that any $\overline{\nabla}\in[\Hat{\nabla}]$ will also satisfy both conditions. Also, if $\Tilde{\nabla}$ is the projection to $\Gamma(\nu\mathcal{F})$ of any connection on $M$ satisfying conditions $(i)$ and $(ii)$, then it is clear that $\Tilde{\nabla}$ is both torsion-free and compatible with the vector bundle metric $\Hat{g}$ induced by $g_\intercal$ on $\nu\mathcal{F}$. Therefore $\Tilde{\nabla}$ is precisely the basic Levi-Civita connection of $g_\intercal$. This establishes the uniqueness part of the theorem.

For the existence, choose any bundle-like metric $g$ associated with $g_\intercal$ with its corresponding Levi-Civita connection $\nabla$. For any $X\in \mathfrak{X}(M)$, we write $X=X^\perp+X^\top$ as usual, for its components of the decomposition $TM=T\mathcal{F}^\perp\oplus T\mathcal{F}$. Define $\Hat{\nabla}$ by
\begin{equation}
\label{biblically accurate angel}
\Hat{\nabla}_XY:=\left(\nabla_XY^\top\right)^\top+[X^\top, Y^\perp]+\left(\nabla_{Y^\perp}X^\top\right)^\top+\left(\nabla_{X^\perp}Y^\perp\right)^\perp.    
\end{equation}
It is clear that $\Hat{\nabla}$ is $ \mathbb{R}$-bilinear. We check it is in fact an affine connection. Pick $f\in C^\infty(M)$, and compute:
\begin{align*}
\Hat{\nabla}_{fX}Y&=\left(\nabla_{fX}Y^\top\right)^\top+[(fX)^\top, Y^\perp]+\left(\nabla_{Y^\perp}(fX)^\top\right)^\top+\left(\nabla_{(fX)^\perp}Y^\perp\right)^\perp\\
&=f\left(\nabla_{X}Y^\top\right)^\top-Y^\perp(f)X^\top+f[X^\top, Y^\perp]+\left(Y^\perp(f)X^\top+f\nabla_{Y^\perp}X^\top\right)^\top+f\left(\nabla_{X^\perp}Y^\perp\right)^\perp\\
&=f\left(\nabla_{X}Y^\top\right)^\top+f[X^\top, Y^\perp]+f\left(\nabla_{Y^\perp}X^\top\right)^\top+f\left(\nabla_{X^\perp}Y^\perp\right)^\perp\\
&=f\Hat{\nabla}_XY,
\end{align*}
and also
\begin{align*}
\Hat{\nabla}_XfY&=\left(\nabla_X(fY)^\top\right)^\top+[X^\top, (fY)^\perp]+\left(\nabla_{(fY)^\perp}X^\top\right)^\top+\left(\nabla_{X^\perp}(fY)^\perp\right)^\perp\\
&=\left((Xf)Y^\top+f\nabla_XY^\top\right)^\top+(X^\top f)Y^\perp+f[X^\top, Y^\perp]+f\left(\nabla_{Y^\perp}X^\top\right)^\top\\
 &\phantom{==} +\left((X^\perp f)Y^\perp+f\nabla_{X^\perp}Y^\perp\right)^\perp\\
&=(Xf)Y^\top+(X^\top f)Y^\perp+(X^\perp f)Y^\perp+f\Hat{\nabla}_XY\\
&=(Xf)Y^\top+(Xf)Y^\perp+f\Hat{\nabla}_XY\\
&=(Xf)Y+f\Hat{\nabla}_XY.
\end{align*}

Next, we verify that $\Hat{\nabla}$ is a transverse affine connection. To see this, choose $\omega_XV:=(\nabla_XV)^\top$ as the partner connection on $T\mathcal{F}$. Then, for any $V\in\Gamma(T\mathcal{F})$ and any $X\in\mathfrak{X}(M)$ we have
$$\Hat{\nabla}_XV=\left(\nabla_XV^\top\right)^\top=\left(\nabla_XV\right)^\top=\omega_XV$$
and 
\begin{align*}
\Hat{\nabla}_VX&=\left(\nabla_VX^\top\right)^\top+[V^\top, X^\perp]+\left(\nabla_{X^\perp}V^\top\right)^\top\\
&=\left(\nabla_VX^\top\right)^\top+[V, X^\perp]+\left(\nabla_{X^\perp}V\right)^\top\\
&=\omega_VX^\top+[V, X^\perp]+\omega_{X^\perp}V\\
&=\omega_{X^\top}V+[V,X^\top]+[V, X^\perp]+\omega_{X^\perp}V\\
&=\omega_{X}V+[V,X].
\end{align*}
Lastly, we argue that
$$\left(\mathcal{L}_V\Hat{\nabla}\right)(X,Y)=[V,\Hat{\nabla}_XY]-\Hat{\nabla}_{[V,X]}Y-\Hat{\nabla}_X[V,Y].$$
Indeed, since it is always true that the Lie derivative of a connection is a tensor, we can check this condition locally, and in particular we can assume that $X, Y$ are projectable vector fields. A simple computation shows that under this assumption the two last terms above will be in $T\mathcal{F}$. Hence, we need only with the first one:
\begin{align*}
[V,\Hat{\nabla}_XY]&=[V, \left(\nabla_XY^\top\right)^\top+[X^\top, Y^\perp]+\left(\nabla_{Y^\perp}X^\top\right)^\top+\left(\nabla_{X^\perp}Y^\perp\right)^\perp].
\end{align*}The bracket of $V$ with the first three terms is clearly in $T\mathcal{F}$, so we are only left to show that $\left(\nabla_{X^\perp}Y^\perp\right)^\perp$ is projectable. But this follows readily from the properties of semi-Riemannian submersions, which we can employ here locally, since $g$ is a bundle-like metric for $(M, \mathcal{F})$.

The final step in the proof is to check the validity of properties $(i)$ and $(ii)$. For $(i)$, we compute
\begin{align*}
\mathrm{Tor}(\Hat{\nabla})(X,Y)&:=\Hat{\nabla}_XY-\Hat{\nabla}_YX-[X,Y]\\
&=\left(\nabla_XY^\top\right)^\top-\left(\nabla_YX^\top\right)^\top+[X^\top,Y^\perp]-[Y^\top,X^\perp]\\
&\phantom{==} +\left(\nabla_{Y^\perp}X^\top\right)^\top-\left(\nabla_{X^\perp}Y^\top\right)^\top+\left(\nabla_{X^\perp}Y^\perp\right)^\perp-\left(\nabla_{X^\perp}Y^\perp\right)^\perp-[X,Y]\\
&=\left(\nabla_{X^\top}Y^\top-\nabla_{Y^\top}X^\top\right)^\top+[X^\top,Y^\perp]-[Y^\top,X^\perp]+[X^\perp, Y^\perp]^\perp-[X,Y]\\
&=[X^\top, Y^\top]^\top+[X^\top,Y^\perp]+[X^\perp,Y^\top]+[X^\perp, Y^\perp]^\perp-[X,Y]\\
&=-[X^\perp, Y^\perp]^\top.
\end{align*}
For $(ii)$ we have
\begin{align*}
g_\intercal(\Hat{\nabla}_XY,Z)+g_\intercal(Y,\Hat{\nabla}_XZ)&=g((\Hat{\nabla}_XY)^\perp,Z^\perp)+g(Y^\perp,(\Hat{\nabla}_XZ)^\perp)\\
&=g([X^\top,Y^\perp]^\perp+(\nabla_{X^\perp}Y^\perp)^\perp, Z^\perp)+g(Y^\perp, [X^\top, Z^\perp]^\perp+(\nabla_{X^\perp}Z^\perp)^\perp)\\
&=g(\nabla_{X^\perp}Y^\perp, Z^\perp)+g(Y^\perp,\nabla_{X^\perp}Z^\perp)\\
&\phantom{==}+g(\nabla_{X^\top}Y^\perp-\nabla_{Y^\perp}X^\top, Z^\perp)+g(Y^\perp, \nabla_{X^\top}Z^\perp-\nabla_{Z^\perp}X^\top)\\
&=X^\perp g(Y^\perp, Z^\perp)+X^\top g(Y^\perp, Z^\perp)-Y^\perp\cancel{g(X^\top, Z^\perp)}-Z^\perp\cancel{g(X^\top, Z^\perp)}\\
&\phantom{==}+g(X^\top, \nabla_{Y^\perp}Z^\perp+\nabla_{Z^\perp}Y^\perp)\\
&=Xg(Y^\perp, Z^\perp)+\cancel{g(X^\top, \nabla_{Y^\perp}Z^\perp+\nabla_{Z^\perp}Y^\perp)}\\
&=Xg_\intercal(Y,Z),
\end{align*}
where the cancellation of the term in the penultimate line is due to the fact that $\nabla$ is the Levi-Civita connection of a bundle-like metric, and therefore $\nabla$-geodesics $g$-orthogonal to $T\mathcal{F}$ at one point remain $g$-orthogonal everywhere, by transnormality. But by the discussion in Ref. \cite{lewis}, this is tantamount to the distribution $T\mathcal{F}^\perp$ being closed with respect to the \textit{symmetric product} $\langle E:F\rangle:=\nabla_EF+\nabla_FE$.
\end{proof}

\begin{remark}\label{transverse-curvature-and-torsion}
As any connection, a transverse affine connection $\Hat{\nabla}$ gives rise to a curvature tensor and a torsion tensor, which we denote respectively by $\Hat{R}$ and $\mathrm{Tor}\Hat{\nabla}$. If $X, Y, Z\in\mathfrak{X}(M)$ and $V, W\in\Gamma(T\mathcal{F})$ we have that:

\begin{align*}
-(\mathrm{Tor}\Hat{\nabla})(V,X)&=(\mathrm{Tor}\Hat{\nabla})(X,V)=\Hat{\nabla}_XV-\Hat{\nabla}_VX-[X,V]\\
&=\omega_XV-[V,X]-\omega_XV-[X,V]=0,
\end{align*}

\begin{align}\label{transverse-curvature-one-vertical-vector}
\begin{split}
\Hat{R}(V,Y)Z&=\Hat{\nabla}_V\Hat{\nabla}_YZ-\Hat{\nabla}_Y\Hat{\nabla}_VZ-\Hat{\nabla}_{[V,Y]}Z\\
&=[V,\Hat{\nabla}_YZ]+\omega_{\Hat{\nabla}_YZ}V-\Hat{\nabla}_Y\left([V,Z]+\omega_ZV\right)-\Hat{\nabla}_{[V,Y]}Z\\
&=\left(\mathcal{L}_V\Hat{\nabla}\right)(Y,Z)+\omega_{\Hat{\nabla}_YZ}V-\omega_Y\omega_ZV
\end{split}
\end{align}
and
\begin{align*}
\Hat{R}(V,W)Z&=\Hat{\nabla}_V\Hat{\nabla}_WZ-\Hat{\nabla}_W\Hat{\nabla}_VZ-\Hat{\nabla}_{[V,W]}Z\\
&=\Hat{\nabla}_V\left([W,Z]+\omega_ZW\right)-\Hat{\nabla}_W\left([V,Z]+\omega_ZV\right)-[[V,W],Z]-\omega_Z[V,W]\\
&=[V,[W,Z]]+\omega_{[W,Z]}V+\omega_V\omega_ZW-[W,[V,Z]]-\omega_{[V,Z]}W\\
&\phantom{==}-\omega_W\omega_ZV-[[V,W],Z]-\omega_Z[V,W]\\
&=\omega_{[W,Z]}V+\omega_V\omega_ZW-\omega_{[V,Z]}W-\omega_W\omega_ZV-\omega_Z[V,W]\\
&=\omega_{[W,Z]}V+\omega_V\omega_ZW-\omega_{[V,Z]}W-\omega_W\omega_ZV-\omega_Z\left(\omega_VW-\omega_WV\right)\\
&=\left(\omega_{[W,Z]}-\omega_W\omega_Z+\omega_Z\omega_W\right)V-\left(\omega_{[V,Z]}-\omega_V\omega_Z+\omega_Z\omega_V\right)W\\
&=R^\omega(V,Z)W-R^\omega(W,Z)V.
\end{align*}

These formulae will soon be useful, when we deal with the transverse analogues of geodesics and Jacobi fields.
\end{remark}

\section{Geodesics, transverse-geodesics and transverse Jacobi fields}\label{section: geodesics}

As mentioned above, the main motivation behind the definitions of bundle-like and transverse affine structure as given in the previous section is they can be adapted to the study of (transverse analogues of) \textit{geodesics}. In order to motivate our definition, and since regular foliations are locally defined by submersions, we start by reviewing here what is already known about geodesics for the specific case of submersions with a transverse affine geometry. We then show how these insights can be generalized to transversely affine foliations.

Blumenthal \cite{blumenthal} studies a submersion $F:M\rightarrow N$ between affine manifolds $(M, \nabla^M)$ and $(N, \nabla^N)$, and calls $F$ an \textit{affine submersion} if the derivative map $F_*:TM\rightarrow TN$ commutes with the parallel transports of $\nabla^M$ and $\nabla^N$. However, Abe and Hasegawa \cite{abe-hasegawa} show that Blumenthal's concept is too restrictive to generalize the well-known notion of (semi-)Riemannian submersion as studied by O'Neill in \cite{oneil}. Indeed, with Blumenthal's definition, \textit{every} geodesic on $(M, \nabla^M)$ is mapped onto a geodesic on $(N, \nabla^N)$, not only those tangent to a distinguished horizontal distribution (in the Riemannian case, the horizontal distribution is of course the distribution normal to the leaves). Abe and Hasegawa introduce instead the notion of an \textit{affine submersion with a horizontal distribution} as a submersion between affine manifolds $(M, \nabla^M)$ and $(N, \nabla^N)$, together with a fixed horizontal distribution $\mathcal{H}\subset TM$ such that
$$\mathcal{H}\left(\nabla^M_{\widetilde{X}}\widetilde{Y}\right)=\widetilde{\left(\nabla^N_XY\right)},$$
for all $X, Y\in\mathfrak{X}(N)$, where the tilde denotes the unique lift of the corresponding vector field which is tangent to $\mathcal{H}$. (If there is no risk of confusion, we shall often use the symbols $\mathcal{H}$ and $\mathcal{V}$ both for the horizontal and vertical distributions themselves, as well as for the corresponding projections of vectors onto them.) 

The existence of such a distribution $\mathcal{H}$ allows one to reproduce O'Neill's well-known theory of (semi-)Riemannian submersions in \cite{oneil}. Indeed, we can define the analogues of the O'Neill tensors $A, T\in\mathcal{I}^1_2(M)$:
\begin{align}\begin{split}\label{generalized-oneill}
T_EF&=\mathcal{H}\left(\nabla^M_{\mathcal{V}E}\mathcal{V}F\right)+\mathcal{V}\left(\nabla^M_{\mathcal{V}E}\mathcal{H}F\right),\\
A_EF&=\mathcal{H}\left(\nabla^M_{\mathcal{H}E}\mathcal{V}F\right)+\mathcal{V}\left(\nabla^M_{\mathcal{H}E}\mathcal{H}F\right).
\end{split}\end{align}

Just in the semi-Riemannian case, it is readily seen that $T_E$ and $A_E$ interchange $\mathcal{H}$ and $\mathcal{V}$ and, moreover, that $T_E=T_{\mathcal{V}E}$ and $A_E=A_{\mathcal{H}E}$. However, since $\nabla^M$ is not necessarily torsion-free, we cannot conclude that $T_UV=T_VU$ for vertical fields $U$ and $V$. Furthermore, since there is in principle no underlying metric with which $\nabla^M$ is compatible, it need not be true that $A_XY=-A_YX$ for horizontal fields $X$ and $Y$, nor does it make sense to discuss the skew-symmetry of $A_E$ and $T_E$. This boils down to the fact that the equations for the horizontal and vertical parts of the acceleration of a curve on $M$ retain terms which vanish on the corresponding equations in O'Neill's work.  However, \cite[Corollary 4.2]{abe-hasegawa} it is still true that the projections of \textit{horizontal} geodesics are always geodesics. Moreover, the tensor $A$ governs how much the submersion deviates from the behavior of a semi-Riemannian submersion:

\begin{proposition}[{\cite[Corollary 4.3]{abe-hasegawa}}]
Let $\pi:(M, \nabla^M)\rightarrow(N,\nabla^N)$ be an affine submersion with horizontal distribution $\mathcal{H}$ such that $A_ZZ=0$ for all horizontal vectors $Z$. Then, every horizontal lift of a geodesic of $N$ is a geodesic of $M$.
\end{proposition}

\begin{proposition}[{\cite[Proposition 4.4]{abe-hasegawa}}]
    Every geodesic of $M$ which has a horizontal tangent vector is always horizontal if and only if $A_ZZ=0$ for any horizontal vector $Z$.
\end{proposition}

Our notion of bundle-like affine structure is the direct generalization of Abe-Hasagawa's notion, and essentially reduces to the latter when discussing globally defined submersions. Now, since the formulae \eqref{generalized-oneill} make use only of the connection on $M$, they can be reproduced \textit{ipsis literis} in the setting of a foliated manifold endowed with a bundle-like affine structure $(\mathcal{H}, \nabla)$, and we shall henceforth always assume this to be the case. We proceed to see how it can be used to obtain generalizations of their results to more general foliated manifolds. 

\begin{remark}\label{reformation-tensor}
Since geodesics and their transverse generalization --- to be defined later on --- will play an important role in the theory developed here, it is important to ensure that they behave just as in the context of semi-Riemannian submersions, where horizontal geodesics project to, and are lifts of, geodesics in the base. As we have previously pointed out, this is governed by the vanishing of the tensor $A$. Therefore, given a distribution $\mathcal{H}$ and a transverse affine connection $\hat{\nabla}$, we introduce the \textit{recalibration tensor (determined by $\hat{\nabla}$ and $\mathcal{H}$)} $\Omega:\mathfrak{X}(M)\times\mathfrak{X}(M)\rightarrow\Gamma(T\mathcal{F})$ given by $$
\Omega(X,Y):=-\frac{1}{2}\mathcal{V}\left(\Hat{\nabla}_{\mathcal{H}X}\mathcal{H}Y+\Hat{\nabla}_{\mathcal{H}Y}\mathcal{H}X\right).$$

Given that $\Omega$ is clearly tensorial, $\overline{\nabla}:=\Hat{\nabla}+\Omega$ defines a new affine connection, and since $\Omega$ is also symmetric, it is easy to see that $\overline{\nabla}$ is again a \textit{transverse} affine connection. In fact, since $\overline{\nabla}$ and $\Hat{\nabla}$ differ by a tensor taking values in the vertical bundle, it follows that $[\overline{\nabla}]=[\Hat{\nabla}]$. Note that if $\overline{A}$ is the $A$-tensor arising from $\mathcal{H}$ and $\overline{\nabla}$, given $H\in\Gamma(\mathcal{H})$ we have that
\begin{align*}    \overline{A}_HH&=\mathcal{V}\left(\overline{\nabla}_HH\right)=\mathcal{V}\left(\Hat{\nabla}_HH+\Omega(H,H)\right)\\&=\mathcal{V}\left(\Hat{\nabla}_HH\right)-\mathcal{V}\left(\Hat{\nabla}_HH\right)\\&=0.
\end{align*}
We thus conclude that given any distribution $\mathcal{H}$ such that $TM=\mathcal{H}\oplus T\mathcal{F}$, there is some $\overline{\nabla}\in[\Hat{\nabla}]$ such that, not only $(\mathcal{H}, \overline{\nabla})$ is a bundle-like affine structure, but it also satisfies the analogue of the so called \textit{transnormality property} of semi-Riemannian foliations: if $\gamma:(a,b)\rightarrow M$ is a $\overline{\nabla}$-geodesic such that for some $t_0\in (a,b)$ the vector $\gamma^\prime(t_0)$ is in $\mathcal{H}_{\gamma(t_0)}$, then $\gamma^\prime(t)\in\mathcal{H}_{\gamma(t)}, \forall t\in(a,b)$.
\end{remark}

\begin{example}
We illustrate the previous consideration via a simple example. On $\mathbb{R}^3$ we fix Cartesian global coordinates $t, x, y$ and $\mathcal{F}$ is the one-dimensional foliation such that $T\mathcal{F}$ is spanned by $\frac{\partial}{\partial y}$. A bundle-like affine structure for $(\mathbb{R}^3,\mathcal{F})$ is given, for example, by $$\left(\mathcal{H}_0=\mathrm{span}\left(\frac{\partial}{\partial t}, \frac{\partial}{\partial x}\right), \nabla^{\mathrm{flat}}\right),$$ where $\nabla^{\mathrm{flat}}$ is the standard flat connection. It is trivial to see that $\nabla^{\mathrm{flat}}$ is a transverse affine connection, and hence it defines a transverse affine structure $[\nabla^{\mathrm{flat}}]$. Conversely, we obtain all other bundle-like affine structures associated with $[\nabla^{\mathrm{flat}}]$ as follows. Note that two vector fields $X=X^t\partial_t+X^x\partial_x+X^y\partial_y$ and $Y=Y^t\partial_t+Y^x\partial_x+Y^y\partial_y$ span a distribution $\mathcal{H}$ complementary to $T\mathcal{F}$ if and only if we have $X^tY^x-X^xY^t\neq 0$ at every point. Suppose this happens, and pick any vector field $Z=Z^t\partial_t+Z^x\partial_x+Z^y\partial_y$. Then we can also write $Z$ as $$
Z=\frac{Z^xY^t-Z^tY^x}{Y^tX^x-Y^xX^t}X+\frac{Z^tX^x-Z^xX^t}{Y^tX^x-Y^xX^t}Y+\left(Z^y-\frac{Z^x(Y^tX^y-X^tY^y)+Z^t(X^xY^y-Y^xX^y)}{Y^tX^x-Y^xX^t}\right)\frac{\partial}{\partial y},
$$so that $$
\mathcal{H}Z=Z^t\frac{\partial}{\partial t}+Z^x\frac{\partial}{\partial x}+\underbrace{\frac{Z^x(Y^tX^y-X^tY^y)+Z^t(X^xY^y-Y^xX^y)}{Y^tX^x-Y^xX^t}}_{=:\tilde{Z}^y}\frac{\partial}{\partial y}.
$$
If $W=W^t\partial_t+W^x\partial_x+W^y\partial_y$, then $$
\nabla^{\mathrm{flat}}_{\mathcal{H}Z}\mathcal{H}W=\left(Z^t\frac{\partial}{\partial t}+Z^x\frac{\partial}{\partial x}+\tilde{Z}^y\frac{\partial}{\partial y}\right)\left(W^t\frac{\partial}{\partial t}+W^x\frac{\partial}{\partial x}+\tilde{W}^y\frac{\partial}{\partial y}\right).
$$
Let $U=U^t\partial_t+U^x\partial_x+U^y\partial_y$ be equal to $\nabla^{\mathrm{flat}}_{\mathcal{H}Z}\mathcal{H}W+\nabla^{\mathrm{flat}}_{\mathcal{H}W}\mathcal{H}Z$, so that \begin{align*}
U^t&=Z^t\frac{\partial W^t}{\partial t}+Z^x\frac{\partial W^t}{\partial x}+\tilde{Z}^y\frac{\partial W^t}{\partial y}+W^t\frac{\partial Z^t}{\partial t}+W^x\frac{\partial Z^t}{\partial x}+\tilde{W}^y\frac{\partial Z^t}{\partial y},\\
U^x&=Z^t\frac{\partial W^x}{\partial t}+Z^x\frac{\partial W^x}{\partial x}+\tilde{Z}^y\frac{\partial W^x}{\partial y}+W^t\frac{\partial Z^x}{\partial t}+W^x\frac{\partial Z^x}{\partial x}+\tilde{W}^y\frac{\partial Z^x}{\partial y},\\
U^y&=Z^t\frac{\partial \tilde{W}^y}{\partial t}+Z^x\frac{\partial \tilde{W}^y}{\partial x}+\tilde{Z}^y\frac{\partial \tilde{W}^y}{\partial y}+W^t\frac{\partial \tilde{Z}^y}{\partial t}+W^x\frac{\partial \tilde{Z}^y}{\partial x}+\tilde{W}^y\frac{\partial \tilde{Z}^y}{\partial y}.
\end{align*}
We thus define $\nabla^\mathcal{H}:=\nabla^{\mathrm{flat}}+\Omega$, where $$
\Omega(Z,W)=-\frac{1}{2}\mathcal{V}U=-\frac{1}{2}\left(U^y-\frac{U^x(Y^tX^y-X^tY^y)+U^t(X^xY^y-Y^xX^y)}{Y^tX^x-Y^xX^t}\right)\frac{\partial}{\partial y}.
$$ For a concrete example, let $X=\partial_t+\cos(t)\partial_y$ and $Y=\partial_x$. Then, we have that \begin{align*}
Y^tX^y-X^tY^y&=0,\\ X^xY^y-Y^xX^y&=-\cos(t),\\ Y^tX^x-Y^xX^t&=-1,
\end{align*}which yields $\tilde{Z}^y=\cos(t)Z^t$. Therefore,
\begin{align*}
    \Omega(Z,W)&=-\frac{1}{2}\left(U^y-\cos(t)U^t\right)\frac{\partial}{\partial y}\\
    &=-\frac{1}{2}\left[Z^t\left(-\sin(t)W^t+\cos(t)\frac{\partial W^t}{\partial t}\right)+\cos(t)Z^x\frac{\partial W^t}{\partial x}+\cos^2(t)Z^t\frac{\partial W^t}{\partial y}\right.\\
&\phantom{==}+W^t\left(-\sin(t)Z^t+\cos(t)\frac{\partial Z^t}{\partial t}\right)+\cos(t)W^x\frac{\partial Z^t}{\partial x}+\cos^2(t)W^t\frac{\partial Z^t}{\partial y}\\
&\phantom{==}-\left.\cos(t)\left(Z^t\frac{\partial W^t}{\partial t}+Z^x\frac{\partial W^t}{\partial x}+\cos(t)Z^t\frac{\partial W^t}{\partial y}+W^t\frac{\partial Z^t}{\partial t}+W^x\frac{\partial Z^t}{\partial x}+\cos(t)W^t\frac{\partial Z^t}{\partial y}\right)\right]\frac{\partial}{\partial y}\\
&=-\frac{1}{2}\left[-2\sin(t)Z^tW^t\right]\frac{\partial}{\partial y}\\
&=\sin(t)Z^tW^t\frac{\partial}{\partial y}.
\end{align*}And thus we have that $$
\nabla^{\mathcal{H}}_{Z}W=\left(Z^t\frac{\partial}{\partial t}+Z^x\frac{\partial}{\partial x}+Z^y\frac{\partial}{\partial y}\right)\left(W^t\frac{\partial}{\partial t}+W^x\frac{\partial}{\partial x}+W^y\frac{\partial}{\partial y}\right)+\sin(t)Z^tW^t\frac{\partial}{\partial y}.
$$
Note that, as expected, we have \begin{align*}
\nabla^\mathcal{H}_XX&=\left(\frac{\partial}{\partial t}+\cos(t)\frac{\partial}{\partial y}\right)\left(\frac{\partial}{\partial t}+\cos(t)\frac{\partial}{\partial y}\right)+\sin(t)\frac{\partial}{\partial y}\\
&=-\sin(t)\frac{\partial}{\partial y}+\sin(t)\frac{\partial}{\partial y}\\
&=0,
\end{align*}which means, in particular, that a curve such as $\gamma(s)=(s+a,b,\sin(s)+c)$ is a $\nabla^\mathcal{H}$-geodesic.
\end{example}

\subsection{Transverse-geodesics}
Let $(M,\f)$ be a foliated manifold. Given a transverse affine connection $\hat{\nabla}$ on $(M,\f)$, a $\Hat{\nabla}$-geodesic is defined in the usual fashion. By Remark \ref{transverse-curvature-and-torsion} the torsion of $\Hat{\nabla}$ vanishes when evaluated on vertical vector fields. This implies that if $\omega:\mathfrak{X}(M)\times\Gamma(T\mathcal{F})\rightarrow\Gamma(T\mathcal{F})$ is a partner connection for $\Hat{\nabla}$, it is also a partner connection for its \textit{symmetric part} $\Hat{\nabla}^{\text{Sym}}:=\Hat{\nabla}-\frac{1}{2}(\mathrm{Tor}\Hat{\nabla})$. To check that the Lie derivative of $\Hat{\nabla}^{\text{Sym}}$ along vertical directions is indeed a tensor taking values in the vertical bundle, it suffices to see that $[V,(\mathrm{Tor}\Hat{\nabla})(X,Y)]$ is vertical for $V\in\Gamma(T\mathcal{F})$. Again, by tensoriality, we can assume that $X, Y$ are projectable, and hence the desired relation readily follows. In conclusion, when dealing with geodesics we may always assume $\Hat{\nabla}$ is torsion-free. However, since the object of geometric significance for us is not $\hat{\nabla}$ \textit{per se}, but the transverse affine structure $[\hat{\nabla}]$ defined by it, we must introduce a new type of curve. 

\begin{definition}
Let $\hat{\nabla}$ be a transverse affine connection on $(M, \mathcal{F})$ and $\gamma:I\rightarrow M$ a smooth curve. We say that $\gamma$ is a \textit{$\hat{\nabla}$-transverse-geodesic} if $\left(\hat{\nabla}_{\gamma^\prime}\gamma^\prime\right)(t)$ is a vertical vector for all $t\in I$.
\end{definition}

\begin{remark}
Ambiguity in terminology notwithstanding, we emphasize that transverse-geodesics should \textit{not} be automatically regarded as ``geodesics which are transverse'', but rather as an atomic term with independent meaning, at least for the time being. In fact, it is easy to see that given the properties of a transverse affine connection, any curve which is contained within some leaf satisfies the definition of a transverse-geodesic. In the case of submersions, any such curve represents a constant geodesics in the base manifold, and as such will be systematically disregarded in most results later on.\end{remark}

\begin{proposition}\label{everything-is-equivalent}
Let $[\hat{\nabla}]$ be a transverse affine structure on $(M, \mathcal{F})$ and let $\gamma:I\rightarrow M$ be a smooth curve. The following statements are equivalent.
\begin{enumerate}
    \item There exists a transverse affine connection $\nabla^\prime\in[\hat{\nabla}]$ such that $\gamma$ is a $\nabla^\prime$-transverse-geodesic;
    \item $\gamma$ is a $\nabla$-transverse-geodesic for all $\nabla\in[\hat{\nabla}]$;
\end{enumerate}If moreover, $\gamma$ admits a distribution $\mathcal{H}$ such that $\mathbb{R}\gamma^\prime\subset\mathcal{H}$ with $TM=T\f\oplus\mathcal{H}$, then the statements above are also equivalent to
\begin{enumerate}
    \item [(3)] There exists a transverse affine connection $\nabla\in[\hat{\nabla}]$ such that $\gamma$ is a $\nabla$-geodesic.
\end{enumerate}
\end{proposition}
\begin{proof}
    $(2)\implies (1)$ and $(3)\implies (1)$ are trivial.

$(1)\implies(2)$. Given any $\nabla\in[\hat{\nabla}]$, by definition, there is a $(1,2)$-tensor field $T$ taking values on $\Gamma(T\mathcal{F})$ such that $\nabla=\nabla^\prime+T$. We thus have that $\nabla_{\gamma^\prime}\gamma^\prime=\nabla^\prime_{\gamma^\prime}\gamma^\prime+T(\gamma^\prime,\gamma^\prime)$, which is a vertical vector everywhere.

$(1)\implies(3)$. Consider the $(1,2)$-tensor field $\Omega$ given by $\Omega(X,Y)=-\frac{1}{2}\mathcal{V}\left(\nabla^\prime_{\mathcal{H}X}\mathcal{H}Y+\nabla^\prime_{\mathcal{H}Y}\mathcal{H}X\right)$ and set $\nabla=\nabla^\prime+\Omega$. We thus have that
\[
\pushQED{\qed}
\nabla_{\gamma^\prime}\gamma^\prime=\nabla^\prime_{\gamma^\prime}\gamma^\prime-\mathcal{V}\left(\nabla^\prime_{\mathcal{H}\gamma^\prime}\mathcal{H}\gamma^\prime\right)=\nabla^\prime_{\gamma^\prime}\gamma^\prime-\nabla^\prime_{\gamma^\prime}\gamma^\prime=0.\qedhere
\popQED
\]
\let\qed\relax
\end{proof} 

\begin{remark}
For each $t_0\in I$, there is $\varepsilon>0$ such that for $\gamma|_{(t_0-\varepsilon,t_0+\varepsilon)}$ satisfies the property of the existence of a distribution $\mathcal{H}$ with $\mathbb{R}\gamma^\prime|_{(t_0-\varepsilon,t_0+\varepsilon)}\subset\mathcal{H}$ and $TM=T\f\oplus\mathcal{H}$. However, this cannot be warranted globally in every situation, as it may happen that there are $t_0,\ldots, t_q\in I$ such that $\gamma(t_0)=\ldots=\gamma(t_q)$ and $\gamma(t_0),\ldots,\gamma(t_q)$ are all linearly independent in $T_{\gamma(t_0)}M$.
\end{remark}

\begin{definition}
A nowhere vertical smooth curve $\gamma:I\rightarrow M$ is said to be a $[\hat{\nabla}]$-\textit{transverse-geodesic} if any, and hence all, of the statements of Proposition \ref{everything-is-equivalent} hold.
\end{definition}

\begin{corollary}\label{geodesics-will-be-geodesics} If a curve $\gamma:I\rightarrow M$ is a $\nabla$-geodesic tangent to $\mathcal{H}$ for some bundle-like affine structure $(\mathcal{H}, \nabla)$, then $\gamma$ is a nowhere vertical $[\Hat{\nabla}]$-transverse-geodesic, where $[\Hat{\nabla}]$ is the transverse affine structure associated to $(\mathcal{H},\nabla)$.
\end{corollary}
\begin{proof}
See $(3)\implies(1)$ in the proof of Proposition \ref{everything-is-equivalent}.
\end{proof}

\begin{remark}
    Given a transverse affine connection $\hat{\nabla}$ with partner connection $\omega$, for every $X, Y\in\Gamma(T\mathcal{F})$ we have that $$
\langle X:Y\rangle :=\hat{\nabla}_{X}Y+\hat{\nabla}_{Y}X=\omega_{X}Y+\omega_{Y}X,
    $$which is a vertical vector field. That means that $T\mathcal{F}$ is a $\hat{\nabla}$-invariant distribution, which as per \cite{lewis}, means that every $\hat{\nabla}$-geodesic tangent to $T\mathcal{F}$ at one point remains tangent to $T\mathcal{F}$ at every point. Thus, we conclude that \textit{the set of all transverse geodesics of a transverse affine structure $[\hat{\nabla}]$ is split between those curves everywhere tangent to $T\mathcal{F}$ (which are therefore each contained in some leaf) and those nowhere tangent to $T\mathcal{F}$.} Of course, the latter class has a more interesting geometric meaning, and we shall restrict our attention to them almost exclusively in what follows.
\end{remark}

\begin{remark}\label{local-diffeo-transv-geo}
If we recall the situation of Proposition \ref{push-pull-submersions}, it is that of a foliation $\mathcal{F}$ given by the connected fibers of a submersion $F:M\rightarrow N$, and a transverse affine structure $[\hat{\nabla}]$ on $(M,\f)$ and an affine connection $\nabla$ on $N$ such that one is the pullback/pushforward of the other. If $\gamma$ is a $[\hat{\nabla}]$-transverse-geodesic on $M$, consider $X\in\mathfrak{L}(\f)$ a vector field that, on some open neighborhood satisfies $X\circ\gamma=\gamma^\prime$. Then $$
\nabla_{(F\circ\gamma)^\prime}(F\circ\gamma)^\prime=(\nabla_{F_*X}F_*X)\circ\gamma=F_*((\hat{\nabla}_{X}X)\circ\gamma)=0,
$$which means that $F\circ\gamma$ is a $\nabla$-geodesic. Conversely, if $\beta$ is any $\nabla$-geodesic on $N$, then any lift of $\beta$ on $M$ is a $[\hat{\nabla}]$-transverse-geodesic.

Meanwhile in the situation depicted in Remark \ref{push-pull-local-diffeo}, we ended up with a foliated affine local diffeomorphism $F$ between $(M,\mathcal{F}^M, [\hat{\nabla}^M])$ and $(N,\mathcal{F}^N, [\hat{\nabla}^N])$, with $\mathcal{F}^M$ and $\mathcal{F}^N$ of the same dimension. Since it is true that $[F^*\hat{\nabla}^N]=[\hat{\nabla}^M]$, and that $F_*$ maps $T\mathcal{F}^M$ isomorphically $T\mathcal{F}^N$, it is easy to see that if $\gamma$ is a $\hat{\nabla}^M$-transverse-geodesic on $M$, then $F\circ\gamma$ is a $\hat{\nabla}^N$-transverse-geodesic on $N$. Conversely, it also holds that given any $\hat{\nabla}^N$-transverse-geodesic $\beta$ on $N$, any lift $\alpha$ of $\beta$ on $M$ is a $\hat{\nabla}^M$-transverse-geodesic.
\end{remark}

In classical affine geometry we know that geodesics exhibit a standard existence and uniqueness property: given $v\in TM$ there is a unique maximal geodesic $\gamma_v$ such that $\gamma^\prime(0)=v$. In the transverse scenario we cannot expect that such a strong result will hold given the degeneracy in the vertical directions. However, the next theorem will provide a kind of uniqueness result for transverse-geodesics modulo projection, as two vectors connected by the differential of a holonomy transformation play the role of a single tangent vector over the leaf space $M/\mathcal{F}$.

\begin{lemma}\label{two-points}
Let $L$ be a leaf of $\mathcal{F}$, two distinct points $p,q\in L$ and $h$ be a holonomy transformation of $\mathcal{F}$ between transversals $T_p$ and $T_q$. Then, there is a open set $\mathcal{U}$ containing $p$ and $q$ and submersion $s:\mathcal{U}\rightarrow \mathbb{R}^q$ defining $\mathcal{F}$ locally, such that $p$ and $q$ lie on the same $s$-plaque and satisfying $s\circ h|_T=s|_{T}$, for some subtranversal $T\ni p$ of $T_p$.
\end{lemma}

\begin{proof}
The transformation $h$ corresponds to some simple curve $\alpha:[0,1]\rightarrow L$ connecting $p$ to $q$. By the compactness of $\alpha[0,1]$ we can finitely cover it by open sets $\mathcal{V}_1,\ldots, \mathcal{V}_k$ with $p\in\mathcal{V}_1, q\in\mathcal{V}_k$ and submersions $s_\lambda:\mathcal{V}_\lambda\rightarrow\mathbb{R}^q, \lambda=1,\ldots, k$, defining $\mathcal{F}$ locally. There is no loss of generality in supposing that $\mathcal{V}_\alpha\cap\mathcal{V}_\beta \neq\emptyset\iff \beta=\alpha\pm1$.

Define $\mathcal{U}_1$ as the ``plaque saturation'' in $\mathcal{V}_1\cup\mathcal{V}_2$ of $\mathcal{V}_1\cap\mathcal{V}_2$. Define $\mathcal{U}_2$ as the plaque saturation in $\mathcal{U}_1\cup\mathcal{V}_3$ of $\mathcal{U}_1\cap\mathcal{V}_3$. Proceed analogously up to $\mathcal{V}_k$ and obtain $\mathcal{U}:=\mathcal{U}_{k-1}$. Define $s:\mathcal{U}\rightarrow \mathbb{R}^q$ by:
$$
s(x)=\begin{cases}
s_1(x),&\text{ if } x\in\mathcal{V}_1,\\
\!\begin{aligned}
s(z), &\text{ for some }z\in\mathcal{V}_{\lambda-1}\setminus\mathcal{V}_{\lambda}\text{ such that }\exists y\in\mathcal{V}_{\lambda-1}\cap\mathcal{V}_\lambda\\&\text{ with } s_\lambda(y)=s_\lambda(x)\text{ and }s_{\lambda-1}(z)=s_{\lambda-1}(y), &\text{ if }x\in\mathcal{V}_\lambda, \lambda>1.   
\end{aligned}
\end{cases}
$$

It is straightforward to check that $s$ is well defined, and its recursive definition ensures that its plaques are unions of the plaques of all of the $s_\lambda$. Moreover, it is clear that it is smooth, locally defines $\mathcal{F}$, and $s(p)=s(q)$. Finally, let $T$ be any subtransversal of $T_p$ such that $T\subset\mathcal{U}$. Since $h$ is precisely defined as sliding along the plaques as one moves along $\alpha$, it is clear that $h(T)\subset\mathcal{U}$ and $h$ preserves the submersion $s$.
\end{proof}

\begin{theorem}
Let $I, J$ be open intervals containing 0, and $\gamma: I\rightarrow M$, $\beta:J\rightarrow M$ two nowhere vertical $[\hat{\nabla}]$-transverse-geodesics. Suppose that there is some holonomy transformation $h:T_p\rightarrow T_q$, where $T_p$ and $T_q$ are transversals respectively at $p$ and $q$ to which $\gamma^\prime(0)$ and $\beta^\prime(0)$ are tangent, such that $dh(\gamma^\prime(0))=\beta^\prime(0)$. Then, there is $\varepsilon>0$ such that $\pi_\mathcal{F}\circ\gamma|_{(-\varepsilon, \varepsilon)}=\pi_\mathcal{F}\circ\beta|_{(-\varepsilon, \varepsilon)}$.
\end{theorem}

\begin{proof}
We summon Lemma \ref{two-points} to obtain a submersion $s:\mathcal{U}\rightarrow\mathbb{R}^q$ defining $\mathcal{F}$ with $\gamma(0), \beta(0)\in\mathcal{U}$. The Lemma also yields that $s\circ h|_T=s|_{T}$, for some subtranversal $T\ni p$ of $T_p$. This means in particular that \begin{equation}\label{equation}
ds_{\gamma(0)}(\gamma^\prime(0))=ds_{\beta(0)}\circ dh_{\gamma(0)}(\gamma^\prime(0))=ds_{\beta(0)}(\beta^\prime(0))=:v. 
\end{equation}
Choose $\varepsilon>0$ such that $\gamma|_{(-\varepsilon, \varepsilon)}$ and $\beta|_{(-\varepsilon, \varepsilon)}$ are contained in $\mathcal{U}$. By Proposition \ref{push-pull-submersions} and Remark \ref{local-diffeo-transv-geo}, $s(\mathcal{U})$ admits an affine connection $\nabla=s_*\hat{\nabla}$ such that $s\circ\gamma|_{(-\varepsilon, \varepsilon)}$ and $s\circ\beta|_{(-\varepsilon, \varepsilon)}$ are $\nabla$-geodesics. But by equation \eqref{equation} both of these projected curves have initial velocity $v$. By the uniqueness of geodesics, it follows that $s\circ\gamma|_{(-\varepsilon, \varepsilon)}=s\circ\beta|_{(-\varepsilon, \varepsilon)}$, and therefore $\pi_\mathcal{F}\circ\gamma|_{(-\varepsilon, \varepsilon)}=\pi_\mathcal{F}\circ\beta|_{(-\varepsilon, \varepsilon)}$.
\end{proof}

\begin{figure}[h]
\centering{
%\resizebox{0.8\textwidth}{!}
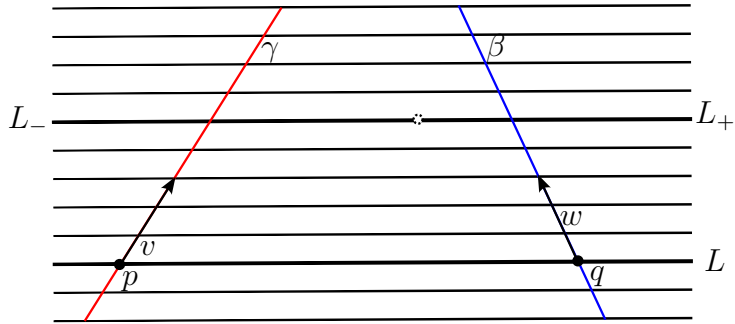
\caption{Geodesics $\gamma$ and $\beta$ have coinciding projection up until they reach leaves $L_-$ and $L_+$,  respectively.}
\label{projecao-das-geodesicas}}
\end{figure}
\subsection{Transverse Jacobi fields}

For a semi-Riemannian submersion $\pi:(M, g)\rightarrow (N, h)$, the equations governing the relation between Jacobi fields on $M$ and on $N$ are found in theorem 2 of \cite{oneill2}, and are given by \begin{align}
\mathcal{V}\left(E^{\prime\prime}+R(E,\gamma^\prime)\gamma^\prime\right)&=\mathcal{V}\left(\mathfrak{D}(E)^\prime\right)+T_{\mathfrak{D}(E)}\gamma^\prime\label{jacobi-submersion1}\\
\mathcal{H}\left(E^{\prime\prime}+R(E,\gamma^\prime)\gamma^\prime\right)&=\left(E_*^{\prime\prime}+R_*(E_*,\sigma^\prime)\sigma^\prime\right)\widetilde{}\ +2A_{\gamma^\prime}(\mathfrak{D}(E))\label{jacobi-submersion2},
\end{align}where $\gamma$ is a horizontal $g$-geodesic, $E\in\mathfrak{X}(\gamma)$, $E_*$ is the projection of $E$ via $d\pi$, $\sigma:=\pi\circ\gamma$ and $\mathfrak{D}(E)=\mathcal{V}((\mathcal{V}E)^\prime)+2A_{\gamma^\prime}\mathcal{H}E-T_{(\mathcal{V}E)}\gamma^\prime$. 

It is our goal in this section to develop a concept of vector fields over curves which would represent an analogue of Jacobi fields on the leaf space $M/\mathcal{F}$, with the particular case where $M/\mathcal{F}=N$ (that is, $\pi$ has connected fibers) being our control scenario. We fix throughout this whole section a nowhere vertical $[\hat{\nabla}]$-transverse-geodesic $\gamma:[a,b]\rightarrow M$ and we will \textit{ad hoc} suppose that $\gamma$ admits a distribution $\mathcal{H}$ such that $\mathbb{R}\gamma^\prime\subset \mathcal{H}$. As usual, we consider the recalibration tensor $\Omega$ defined by $\hat{\nabla}$ and $\mathcal{H}$ and set $\nabla:=\Hat{\nabla}+\Omega$. By Proposition \ref{everything-is-equivalent}, we know that $\gamma$ is a $\nabla$-geodesic.

\begin{remark}\label{remark-on-submersions}
Even though equations \eqref{jacobi-submersion1} and \eqref{jacobi-submersion2} are not \textit{a priori} valid in our setting, note that since $\nabla$ is chosen such that $A_HH=0$, the tensors $A$ and $T$ here still have the same properties of those in the semi-Riemannian submersion setting. We thus have the validity of at least equation \eqref{jacobi-submersion1}, since it does not refer to any objects in the base --- unlike equation \eqref{jacobi-submersion2}.
\end{remark}

Now, suppose $J\in\mathfrak{X}(\gamma)$ is any vector field. We denote by a prime the covariant derivative with respect to $\nabla$, whereas the one with respect to $\Hat{\nabla}$ will be denoted by a dot. A straightforward computation will lead us to:
\begin{align}\label{relation-between-two-jacobi-equations}\begin{split}
\mathcal{V}\left(J^{\prime\prime}+R(J,\gamma^\prime)\gamma^\prime\right)&=\mathcal{V}\left(\ddot J+\Hat{R}(J, \gamma^\prime)\gamma^\prime\right)-2\Omega(\gamma^\prime,[J,\gamma^\prime])\\
\mathcal{H}\left(J^{\prime\prime}+R(J,\gamma^\prime)\gamma^\prime\right)&=\mathcal{H}\left(\ddot J+\Hat{R}(J, \gamma^\prime)\gamma^\prime\right)\end{split}\end{align}

In the case $\mathcal{F}$ is given by the connected fibers of a submersion, it is revealing to compare the second equation above with equation \eqref{jacobi-submersion2}. Doing so will show that:\begin{align}\label{horizontal-part-of-transverse-jacobi-equation}\begin{split}
\mathcal{H}\left(\ddot J+\Hat{R}(J, \gamma^\prime)\gamma^\prime\right)=\left(J_*^{\prime\prime}+R_*(J_*,\sigma^\prime)\sigma^\prime\right)\widetilde{}\ .
\end{split}
\end{align}Therefore $J_*$ is a Jacobi field on $N$ if and only if $J$ is a transverse Jacobi field in the sense of the following definition.

\begin{definition}
Let $\gamma:I\rightarrow M$ be a non-vertical $\Hat{\nabla}$-transverse-geodesic and $J\in\mathfrak{X}(\gamma)$. We say that $J$ is a $\hat{\nabla}$-\textit{transverse Jacobi field} if for every $t\in I$ we have that $$
\hat{\mathcal{J}}(J):=\ddot J(t)+\Hat{R}_{\gamma(t)}\left(J(t), \gamma^\prime(t)\right)\gamma^\prime(t)
$$is a vertical vector.
\end{definition}

If $\tilde{\nabla}\in[\hat{\nabla}]$, then $\hat{\nabla}=\tilde{\nabla}+T$ for some $(1,2)$-tensor $T$ taking values in $T\mathcal{F}$. Therefore, it is easily seen that given a $\hat{\nabla}$-transverse Jacobi field $J$, $\tilde{\mathcal{J}}(J)$ will also be everywhere vertical, which in turn means that being a transverse Jacobi field is a property valued for any representative of the class $[\hat{\nabla}]$. Thus we say that $J$ is a $[\hat{\nabla}]$-transverse Jacobi field when it is a $\tilde{\nabla}$-transverse Jacobi field for some (and hence for any) $\tilde{\nabla}\in[\hat{\nabla}]$.

\begin{remark}\label{consequences-transverse-jacobi-field}
Equations \eqref{relation-between-two-jacobi-equations} show in particular that every Jacobi field is a transverse Jacobi field. Also, if $J$ is a transverse Jacobi field it is clear that given any distribution $\mathcal{H}$ we must have $\mathcal{H}\left(\ddot J+\Hat{R}(J, \gamma^\prime)\gamma^\prime\right)\equiv0$. Thus, in view of equation \eqref{horizontal-part-of-transverse-jacobi-equation}, from now on we understand transverse Jacobi fields as representations of what would be Jacobi fields on $M/\mathcal{F}.$ Moreover, if $V\in\Gamma(T\mathcal{F})$, employing equation \eqref{transverse-curvature-one-vertical-vector}, we have that \begin{align*}
\ddot V+\Hat{R}(V,\gamma^\prime)\gamma^\prime&=\Hat{\nabla}_{\gamma^\prime}\Hat{\nabla}_{\gamma^\prime}V+\left(\mathcal{L}_V\Hat{\nabla}\right)(\gamma^\prime,\gamma^\prime)+\omega_{\cancel{\Hat{\nabla}_{\gamma^\prime}\gamma^\prime}}V-\omega_{\gamma^\prime}\omega_{\gamma^\prime}V\\
&=\cancel{\omega_{\gamma^\prime}\omega_{\gamma^\prime}V}+\left(\mathcal{L}_V\Hat{\nabla}\right)(\gamma^\prime,\gamma^\prime)-\cancel{\omega_{\gamma^\prime}\omega_{\gamma^\prime}V}\\
&=\left(\mathcal{L}_V\Hat{\nabla}\right)(\gamma^\prime,\gamma^\prime),
\end{align*}which is a vertical vector field, by the definition of transverse affine connection. Thus we have that any vertical vector field is a transverse Jacobi field. In particular, if $J$ is a transverse Jacobi field, then $\mathcal{H}J=J-\mathcal{V}J$ is also a transverse Jacobi field.

Equivalently, we can deal with a \textit{normal Jacobi structure} $\overline{J}\in\Gamma(\gamma^*\nu\mathcal{F})$, where $J$ is a transverse Jacobi field on $\gamma$. Recall that $\Hat{\nabla}$ induces a connection $\overline{\nabla}:\mathfrak{X}(M)\times\Gamma(\nu\mathcal{F})\rightarrow \Gamma(\nu\mathcal{F})$, which has its curvature tensor $\overline{R}:\mathfrak{X}(M)\times\mathfrak{X}(M)\times\Gamma(\nu\mathcal{F})\rightarrow \Gamma(\nu\mathcal{F})$ defined by $$
\overline{R}(X,Y)\overline{Z}:=\overline{\hat{R}(X,Y)Z}.
$$
Then, for any $\overline{J}\in\Gamma(\gamma^*\nu\mathcal{F})$, we have that
$$
\overline{\mathcal{J}}(\overline{J})=\overline{\nabla}_{\gamma^\prime}\overline{\nabla}_{\gamma^\prime}\overline{J}+\overline{R}(J,\gamma^\prime)\overline{\gamma^\prime}=\overline{\hat{\nabla}_{\gamma^\prime}\hat{\nabla}_{\gamma^\prime}J+\hat{R}(J,\gamma^\prime)\gamma^\prime}=\overline{\hat{\mathcal{J}}(J)},
$$which means that $\overline{J}$ also satisfies a Jacobi-type equation on $\nu\mathcal{F}$ when $J$ is transverse Jacobi.

If we denote by $\mathcal{J}_\intercal(\gamma)$ the vector space of transverse Jacobi fields along $\gamma$ and by $\mathcal{J}_\nu(\gamma)$ the vector space of normal Jacobi structures on $\gamma$, it is clear by the above remark that $$
\frac{\mathcal{J}_\intercal(\gamma)}{\Gamma(\gamma^*T\mathcal{F})}\cong\mathcal{J}_\nu(\gamma).
$$
When a distribution $\mathcal{H}$ is given, we define $\mathcal{J}_{\intercal,\mathcal{H}}(\gamma):=\{\mathcal{H}J: J\in\mathcal{J}_\intercal(\gamma)\}\subset \mathcal{J}_\intercal(\gamma)$, and note that, by construction we have\begin{align*}
    \mathcal{J}_\intercal(\gamma)&=\mathcal{J}_{\intercal,\mathcal{H}}(\gamma)\oplus \Gamma(\gamma^*T\mathcal{F}),
\end{align*}which means that $\mathcal{J}_{\intercal,\mathcal{H}}(\gamma)\cong\mathcal{J}_\nu(\gamma)$.
\end{remark}
Following O'Neill, we also define the following vector spaces, which will be important in the end of the section: \begin{align*}
\overline{\mathcal{J}}^{a,b}_\intercal(\gamma)&:=\{J\in \overline{\mathcal{J}}_\intercal(\gamma): J(a)\text{ and }J(b)\text{ are vertical}\}\\
\mathcal{J}^{a,b}_{\intercal, \mathcal{H}}(\gamma)&:=\{J\in \mathcal{J}_{\intercal,\mathcal{H}}(\gamma): J(a)\text{ and }J(b)\text{ are zero}\}\\
\mathcal{J}^{v0}_\delta(\gamma)&:=\{J\in\mathcal{J}(\gamma): \mathfrak{D}(J)=0, J(a)\text{ is vertical, and } J(b)=0\}\\
\mathcal{J}^{0v}_\delta(\gamma)&:=\{J\in\mathcal{J}(\gamma): \mathfrak{D}(J)=0, J(b)\text{ is vertical, and } J(a)=0\}\\
\mathcal{J}^{vv}_\delta(\gamma)&:=\{J\in\mathcal{J}(\gamma): \mathfrak{D}(J)=0, J(a)\text{ and } J(b) \text{ is vertical}\}.
\end{align*}

We end this section by showing that with the technology of transverse Jacobi fields we are still able to relate conjugate points ``in the base'' with focal points of leaves, as the classic theory of O'Neill did in the special case of a global submersion.

\begin{proposition}\label{given-transverse-jacobi-field-there-is-jacobi-field}
Let $\gamma:I\rightarrow M$ be a non-vertical $[\Hat{\nabla}]$-transverse-geodesic, $v\in T_{\gamma(a)}\mathcal{F}_{\gamma(a)}$ and $\mathcal{H}$ be any complementary distribution which contains $\mathbb{R}\gamma^\prime$. Let $\nabla$ be an affine connection associated to $[\Hat{\nabla}]$ such that $(\mathcal{H}, \nabla)$ is a bundle-like affine structure. If $J\in\mathfrak{X}(\gamma)$ is any transverse Jacobi field tangent to $\mathcal{H}$ with $J(a)=0$, then there exists a unique Jacobi field $E\in\mathfrak{X}(\gamma)$ such that \begin{enumerate}
        \item $\mathcal{H}E=J$;
        \item $\mathfrak{D}(E)=0$;
        \item $E(a)=v$.
    \end{enumerate}
\end{proposition}

\begin{proofcomment}
This is entirely analogous to the proof of Lemma 1 of \cite{oneill2}, rewritten to the language of transverse affine geometry. That is, we deal with the transverse analogue of whichever object O'Neill considers in the target of the submersion.
\end{proofcomment}

\begin{definition}
Let $\gamma:I\rightarrow M$ be a non vertical $\hat{\nabla}$-transverse-geodesic in $M$. Two points $\gamma(a), \gamma(b)$ for $a\neq b\in I$ are \textit{transversely conjugate along} $\gamma$ if there exists a transverse Jacobi field $J\in\mathfrak{X}(\gamma)$ which is not everywhere vertical but such that $J(a)$ and $J(b)$ are vertical vectors. In the affirmative case, the \textit{transverse conjugacy index} of the pair is $$
\text{Conj}_\intercal(\gamma(a), \gamma(b)):=\text{dim}\left(\overline{\mathcal{J}}^{a,b}_\intercal(\gamma)\right), 
$$where $\overline{\mathcal{J}}^{a,b}_\intercal(\gamma)=\{\overline{J}\in\overline{\mathcal{J}}_\intercal(\gamma):J(a), J(b)$ are vertical.$\}$.
\end{definition}

\begin{lemma}\label{this-is-lemma-2-from-ONEILL}
Let $E\in\mathfrak{X}(\gamma)$. Then $E$ is the solution of $$
\begin{cases}
    E^\prime=A_{\gamma^\prime}E+T_E\gamma^\prime\\
    E(a) \mbox{ is vertical}
\end{cases}
$$ if, and only if $E$ is vertical and $\mathfrak{D}(E)=0$.
\end{lemma}
\begin{proofcomment}
Again, this is just a relatively straightforward adaption of Lemma 2 of \cite{oneill2} to the transverse affine geometry parlance.
\end{proofcomment}

\begin{corollary}\label{unfolding-of-lemma-2-from-oneill}
If $J\in\mathcal{J}^{0v}_\delta(\gamma)$ [resp. in $\mathcal{J}^{v0}_\delta(\gamma)$] is such that $\mathcal{H}J=0$, then $J$ is 0.
\end{corollary}
\begin{proof}
The ``if part'' of Lemma \ref{this-is-lemma-2-from-ONEILL} yields that such $J$ satisfies a homogeneous first order linear equation. Then $J(a)=0$ [resp. $J(b)=0$] implies that $J\equiv 0$.
\end{proof}

\begin{remark}
    Recall that on a semi-Riemannian manifold, a point $\gamma(b)$ is a focal point of a submanifold $P$ along a geodesic $\gamma$ that is normal to $P$ if there exists a nonzero $P$-Jacobi field $J$ along $\gamma$ such that $J(b)=0$. Now, being $P$-Jacobi means that $J(a)\in T_{\gamma(a)}P$ and $\tan_P(J^\prime(a))=-S_{\gamma^\prime(a)}J(a)$, where $S_{\gamma^\prime(a)}:T_{\gamma(a)}P\rightarrow T_{\gamma(a)}P$ is the Weingarten operator of $P$ at $\gamma(a)$ associated with the normal vector $\gamma^\prime(a)$, which is given by $S_{\gamma^\prime(a)}v=-\nabla_v\gamma^\prime$. Even though these notions are not available in our current setting, if $J$ is vertical and $\gamma^\prime$ is tangent to $\mathcal{H}$, we have that $T_J\gamma^\prime=\mathcal{V}(\nabla_J\gamma^\prime)$, and therefore we can say that $J$ is $\mathcal{F}_{\gamma(a)}$-Jacobi if $J(a)\in T_{\gamma(a)}\mathcal{F}_{\gamma(a)}$ and $\mathcal{V}J^\prime(a)=T_{J(a)}\gamma^\prime(a)$, or equivalently $\mathfrak{D}(J)(a)=0$. However, equation \eqref{jacobi-submersion1} implies that $\mathcal{V}(\mathfrak{D}(J))^\prime=-T_{\mathfrak{D}(J)}\gamma^\prime$ and then $(\mathfrak{D}(J))^\prime=-T_{\mathfrak{D}(J)}\gamma^\prime+A_{\gamma^\prime}\mathfrak{D}(J)$, and therefore $\mathfrak{D}(J)$ vanishing at any point means that it is identically zero. Therefore, note that $J$ is a $\mathcal{F}_a$-Jacobi field which vanishes at $b$ if, and only if, $J$ is in $\mathcal{J}^{v0}_\delta(\gamma)$. We thus take the dimension of $\mathcal{J}^{v0}_\delta(\gamma)$ to be the \textit{focal order} of $\gamma(b)$ with respect to $\mathcal{F}_a$ along $\gamma$.
\end{remark}

\begin{remark}\label{transverse-conjugate-points-imply-focal-points}
If $\mathcal{H}$ which contains $\mathbb{R}\gamma^\prime$, then for $E$ in $\mathcal{J}^{vv}_\delta(\gamma)$, (or in $\mathcal{J}^{v0}_\delta(\gamma)$ or in $\mathcal{J}^{0v}_\delta(\gamma)$), we have that $\mathcal{H}E$ is in $\mathcal{J}^{a,b}_{\intercal, \mathcal{H}}(\gamma)$. By Proposition \ref{given-transverse-jacobi-field-there-is-jacobi-field}, this assignment is surjective. Therefore $\mathcal{J}^{a,b}_{\intercal, \mathcal{H}}(\gamma)$ is isomorphic to any of these spaces modulo the kernel of this assignment, which are obviously the vertical Jacobi fields. But, by Corollary \ref{unfolding-of-lemma-2-from-oneill}, this kernel consists only of the zero field, in the case of $\mathcal{J}^{v0}_\delta(\gamma)$ and $\mathcal{J}^{0v}_\delta(\gamma)$. Therefore we have $$
\overline{\mathcal{J}}^{a,b}_\intercal(\gamma)\cong\mathcal{J}^{a,b}_{\intercal, \mathcal{H}}(\gamma)\cong\mathcal{J}^{v0}_\delta(\gamma)\cong\mathcal{J}^{0v}_\delta(\gamma)\cong\frac{\mathcal{J}^{vv}_\delta(\gamma)}{\mathcal{J}_{\delta,v}(\gamma)}.
$$From that we conclude that the following numbers are equal:
\begin{enumerate}
\item The order of $\gamma(b)$ as a focal point of $\mathcal{F}_{\gamma(a)}$ along $\gamma$;
\item The order of $\gamma(a)$ as a focal point of $\mathcal{F}_{\gamma(b)}$ along $\gamma$;
\item The dimension of $\frac{\mathcal{J}^{vv}_\delta(\gamma)}{\mathcal{J}_{\delta,v}(\gamma)}$;
\item The transverse conjugacy index of the pair $\gamma(a), \gamma(b)$ along $\gamma$.
\end{enumerate}
\end{remark}

\section{Main results}\label{sec5}

On an affine manifold $(M, \nabla)$, let $\mathcal{D}\subset TM$ be the maximal domain of the exponential map $\exp =\exp^\nabla$ and let $C\subset \mathcal{D}$ denote a subset of initial velocities of a family of geodesics $\mathcal{C}$ satisfying certain conditions laid out in Ref. \cite{kledilson}, and $\mathcal{Z}\subset TM$ be the image of the zero section.

\begin{theorem}[{\cite[Theorem 3.5]{kledilson}}]\label{pc+disp=proper}
If $\exp|_C$ is a proper map, then the collection $\mathcal{C}$ is pseudoconvex and disprisoning. If $C\cap\mathcal{Z}$ is compact and $C$ is closed in $\mathcal{D}$, then the converse holds.
\end{theorem}

Armed with that result, we go back to our setting and fix, from now on, a general transversely affine manifold $(M,\mathcal{F}, [\Hat{\nabla}])$. We first recall a definition first introduced in \cite{transverse-diameter}.

\begin{definition}\label{def: inextendibility}
Let $\gamma:[a,b)\rightarrow M$ be a continuous curve. We say that a leaf $L\in\mathcal{F}$ is a \textit{transverse limit to the right} for $\gamma$ if for any saturated open set $\mathcal{U}\supset L$ there exists some $s_0\in[a,b)$ such that $s\geq s_0\implies \gamma(s)\in \mathcal{U}$. We say that $\gamma$ is \textit{transversely right-inextendible} if there are no transverse limits to the right for $\gamma$. We analogously define versions to the left for a curve defined on $(a, b]$. Finally, a curve $\gamma:(a,b)\rightarrow M$ is \textit{transversely inextendible} if for some (and hence for any) $c\in (a,b)$ we have both that $\gamma|_{[c,b)}$ is transversely right-inextendible and $\gamma|_{(a,c]}$ is transversely left-inextendible. 
\end{definition}

Note that any curve that is transversely right-inextendible is in particular right-inextendible (in the usual sense), but the converse is not true. Indeed, it is clear that the definition of transverse (right- or left-)inextendibility is equivalent to the requirement that the projection of the curve onto the leaf space $M/\f$ is (right- or left-)inextendible in the usual sense therein. But a curve on $M$ may be inextendible and yet have an extendible projection. With that, we can introduce the transverse analogues of the two key hypotheses in Beem and Parker's version of Hadamard's theorem:

\begin{definition}
Let $\mathcal{C}$ be a family on non-vertical $[\Hat{\nabla}]$-transverse-geodesics. We say that $\mathcal{C}$ is
\begin{itemize}
    \item \textit{transversely pseudoconvex} if for every transversely compact subset $K\subset M$ there is a transversely compact subset $K^*\subset M$ such that if a segment of $\gamma\in\mathcal{C}$ has its endpoints in $K$, then that segment is entirely contained in $K^*$;
    \item \textit{transversely disprisoning} if for any transversely inextendible transverse-geodesic  $\gamma:(a,b)\rightarrow \mathcal{O}$ of a  $(-\infty\leq a<b\leq+\infty)$ in $\mathcal{C}$, and any $t_0\in(a,b)$, neither $\overline{\gamma[t_0,b)}$ nor $\overline{\gamma(a,t_0]}$ is transversely compact. If $\mathcal{C}$ is not transversely disprisoning, then it is said to be \textit{transversely imprisoning}.
\end{itemize}
\end{definition}

Remember from Remark \ref{good-new-switcheroo} that for any complementary distribution $\mathcal{H}$ we obtain a bundle-like affine structure $(\mathcal{H},\nabla)$ associated with $[\Hat{\nabla}]$. We thus define $\mathcal{C}_\mathcal{H}$ as the family of non-constant $\nabla$-geodesics which are tangent to $\mathcal{H}$.

\begin{proposition}\label{pseudoconvexity-rises}
If $(M, \mathcal{F}, [\Hat{\nabla}])$ is transversely pseudoconvex and the leaves of $\mathcal{F}$ are compact, then $\mathcal{C}_\mathcal{H}$ is pseudoconvex.
\end{proposition}
\begin{proof}
If $K$ is a compact subset of $M$, it is in particular transversely compact. Let $K^*$ be the transversely compact subset which contains all of the $\Hat{\nabla}$-transverse-geodesic segments whose endpoints lie in $K$. Since the leaves of $\mathcal{F}$ are compact, it follows that $K^*$ is compact. By Corollary \ref{geodesics-will-be-geodesics} we have that any $\nabla$-geodesic tangent to $\mathcal{H}$ is a $\Hat{\nabla}$-transverse-geodesic, and therefore $\mathcal{C}_\mathcal{H}$ is pseudoconvex.
\end{proof}

\begin{remark}\label{suf-not-nec}
We take a moment to emphasize that the compactness of the leaves of $\mathcal{F}$ is a sufficient but far from necessary hypothesis for the proposition above. To see that, consider for example the simple foliation of $\mathbb{R}^3$ given by the submersion $(t,x,y)\mapsto(t,x)$. The flat connection on $\mathbb{R}^3$ defines a transverse affine structure and also a bundle-like affine structure together with the distribution $\mathcal{H}=\mathbb{R}\frac{\partial}{\partial t}\oplus\mathbb{R}\frac{\partial}{\partial x}$. It is clear that this example is transversely pseudoconvex and $\mathcal{C}_\mathcal{H}$ is pseudoconvex, even though the leaves are homeomorphic to $\mathbb{R}$.   
\end{remark}

\begin{proposition}\label{disprisonment-rises}
If $(M, \mathcal{F}, [\Hat{\nabla}])$ is transversely disprisoning and $M/\mathcal{F}$ is Hausdorff, then $\mathcal{C}_\mathcal{H}$ is disprisoning.    
\end{proposition}
\begin{proof}
Let $\gamma:(a,b)\rightarrow M$ ($-\infty\leq a<b\leq +\infty$) be any $\nabla$-geodesic tangent to $\mathcal{H}$ and suppose that for some $t_0\in(a,b)$ either $\gamma(a, t_0]$ or $\gamma[t_0, b)$ are contained in some compact subset $K\subset M$. We shall deal with the case of disprisonment to the right, since the one to the left is entirely analogous. We must show then that $\gamma|_{[t_0,b)}$ is extendible to the right. By Corollary \ref{geodesics-will-be-geodesics}, $\gamma$ is a $\hat{\nabla}$-transverse-geodesic, and therefore the transverse disprisonment condition implies that $\gamma$ has a right endleaf $L$ of $\mathcal{F}$. 

If $(x_k)_{k\in\mathbb{N}}\subset [t_0, b)$ is a sequence converging to $b$, by the compactness of $K$ we have a converging subsequence for $(\gamma(x_k))_{k\in\mathbb{N}}$. If every subsequence converges to the same point $p$, then we would have that $p$ is an endpoint for $\gamma$. So, we suppose that there are at least two subsequences of $(x_k)_{k\in\mathbb{N}}$, say, $t_k$ and $s_k$ converging respectively to distinct points $p$ and $q$ on $L$. There is no loss in generality if we assume that $t_1<s_1<t_2<s_2<\ldots$. Let $h$ be an auxiliary complete Riemannian metric on $M$ and let $\Tilde{\gamma}_k$ be the reparametrization of the restriction of $\gamma_{[t_k, s_k]}$ to an interval $[0, a_k]$ such that $h\left(\Tilde{\gamma}_k^\prime(0),\Tilde{\gamma}_k^\prime(0)\right)=1, \forall k$.

If $v_k=\Tilde{\gamma}_k^\prime(0)$, we have that $\Tilde{\gamma}_k(t)=\exp^\nabla_{\Tilde{\gamma}_k(0)}(tv_k),$ for $t\in[0,a_k]$. Since $\{v_k\}$ lies in a compact, up to subsequence, $v_k$ converges to some $v$ with $h(v,v)=1$ and given that $\mathcal{H}$ is complementary to the tangent bundle of a regular foliation, it is closed, and therefore $v\in\mathcal{H}$. Now, there are two cases to be dealt with.

\noindent\textbf{Case 1.} $(a_k)_{k\in\mathbb{N}}$ is unbounded. Then, we can assume $a_k\rightarrow+\infty$ and therefore, for any $t\in[0,+\infty)$ eventually $t<a_k$, hence $tv_k\in\mathcal{D}$. Therefore $\gamma_k$ converges pointwise to a curve $\tilde{\gamma}$ given by $\tilde{\gamma}(t)=\exp(tv_0)$, which means $\tilde{\gamma}$ is tangent to $\mathcal{H}$.

\noindent\textbf{Case 2.} $(a_k)_{k\in\mathbb{N}}$ is bounded. In this case, we can assume $a_k\rightarrow a$, for some $a\geq0$. However, $a$ cannot be 0, since that would mean that the limit curve would collapse to a point, which cannot happen since we assumed $\gamma(t_k)$ and $\gamma(s_k)$ converge to distinct points. Therefore, $a>0$ and then, for any $t\in[0,a)$, eventually $t<a_k$, and again we conclude that $\tilde{\gamma}$ is tangent to $\mathcal{H}$.

Now, for each $k$ let $\beta_k$ be a reparametrization of $\gamma_k$ to the interval $[0,1]$, and $\beta:[0,1]\rightarrow M$ a reparametrization of $\tilde{\gamma}$. If $\pi_\mathcal{F}:M\rightarrow M/\mathcal{F}$ is the projection, then we have $(\pi_\mathcal{F}\circ\beta_k)(t)\rightarrow (\pi_\mathcal{F}\circ\beta)(t),\forall t\in [0,1]$. On the other hand, if $\Hat{L}$ is the class of the leaf $L$ in $M/\mathcal{F}$, we show that $\pi_\mathcal{F}\circ\beta_k(t)\rightarrow\Hat{L}, \forall t\in[0,1]$. Indeed, given any neighborhood $\Hat{\mathcal{U}}\ni\hat{L}$, $\pi_{\mathcal{F}}^{-1}(\Hat{\mathcal{U}})$ is a saturated neighborhood of $L$. Since $L$ is a right endleaf for $\gamma$, there exists $K\in[t_0,b)$ such that $t\geq K\implies \gamma(t)\in\pi_{\mathcal{F}}^{-1}(\Hat{\mathcal{U}})$. In terms of the sequence $\beta_k $, this means that there is $N\in\mathbb{N}$ such that $k\geq N\implies \mathrm{im}(\beta_k)\subset\pi_{\mathcal{F}}^{-1}(\Hat{\mathcal{U}})$. Since $M/\mathcal{F}$ is Hausdorff, we conclude that $(\pi_\mathcal{F}\circ\beta)(t)=\Hat{L}$, which means that $\beta$ is contained in $L$, but that contradicts the fact that $\beta^\prime\parallel\tilde{\gamma}^\prime\in\mathcal{H}$.
\end{proof}

\begin{theorem}\label{thrm princ}
Let $(M,\mathcal{F})$ be a foliation whose leaves are all compact and such that $M/\mathcal{F}$ is Hausdorff. Suppose $(M,\mathcal{F}$) is endowed with a transverse affine structure $[\Hat{\nabla}]$ such that the family of all non vertical transverse-geodesics is transversely pseudoconvex, transversely disprisoning and without transverse conjugate points. Then the universal cover $\tilde{M}$ of $M$ splits diffeomorphically as a product $ \tilde{L}\times B$, where $B$ is a contractible manifold and $\tilde{L}$ is the universal cover of some leaf $L$ of $\mathcal{F}$.
\end{theorem}

\begin{proof}
%We begin by verifying that the lifted foliation $\pi^*\mathcal{F}$ of $\tilde{M}$ has the \textit{leaf-regularity property\footnote{We employ \textit{ad hoc} this name here to designate what the geometric literature of yore (\textit{e.g.} \cite{reinhart}, \cite{palais}) used to call ``regular foliation'', since the latter term has become mainstream for foliations with leaves of constant dimensions.}}. Suppose that it was not the case. Then, there would be a $q\in\tilde{M}$ such that for any foliated neighborhood $q\in\mathcal{U}$ there is a leaf $L^\prime\in\pi^*\mathcal{F}$ such that $\mathcal{U}\cap L^\prime$ has at least two connected components, \textit{i.e.}, plaques.

Let $\pi:\Tilde{M}\rightarrow M$ be the universal covering map and denote by $\mathcal{F}^*$ the pullback foliation $\pi^*\mathcal{F}$. We will show that $\tilde{M}/\mathcal{F}^*$ is locally Euclidean. Indeed, let $p$ be any point of $\tilde{M}$, denote by $L^\prime$ the leaf of $\mathcal{F}^*$ passing through $p$, and let $L$ be the leaf $\pi(L^\prime)$ on $(M, \mathcal{F})$.
\begin{comment}
Since $\mathcal{F}$ has compact leaves and $M/\mathcal{F}$ is Hausdorff, it is known (see \cite[Thm. 1.4]{millett}) that $M/\mathcal{F}$ has a natural orbifold structure, which in turn means that every leaf of $\mathcal{F}$ has finite holonomy. Now, if we cover $L$ by a finite number of evenly covered neighborhoods $\mathcal{U}_1,\ldots, \mathcal{U}_k$, then, by Reeb stability theorem, there is a saturated tubular neighborhood $\chi:\mathrm{Tub}(L)\rightarrow L$ which can be taken satisfying $\mathrm{Tub}(L)\subset\bigcup_{i=1}^k\mathcal{U}_i$. Furthermore, $\mathcal{F}$ restricted to $\mathrm{Tub}(L)$ is given by the suspension of the holonomy homomorphism $h:\pi_1(L,x)\rightarrow \mathrm{Hol}_x(L)$, for any $x\in L$, and $\chi$ restricted to each leaf inside $\mathrm{Tub}(L)$ is a covering map \cite[Thm. 2.6]{alex4}. For each $i=1,\ldots, k$ let us write $\pi^{-1}(\mathcal{U}_i\cap \mathrm{Tub}(L))=\dot\bigcup_{j=1}^l\mathcal{W}_i^j$ and suppose, without loss of generality, that $L^\prime\subset\bigcup_{i=1}^k\mathcal{W}^1_i=:\mathcal{W}$. Note that $\pi|_{\mathcal{W}_i^j}:\mathcal{W}_i^j\rightarrow \mathcal{U}_i\cap \mathrm{Tub}(L)$ is a diffeomorphism for each $i, j$, and therefore for each leaf of $\tilde{\mathcal{F}}$ inside of $\mathcal{W}$ we can use $\pi$ and $\chi$ to obtain a covering map from it onto $L^\prime$.$$
\xymatrix{
\mathcal{W}\ar@{->}[r]^-{\pi}\ar@{-->}[d]&\mathcal{U}_i\cap \mathrm{Tub}(L)\ar@{->}[d]^-{\chi}\\
L^\prime\ar@{->}[r]_-{\pi|_{L^\prime}}&L
}$$
\end{comment}
Since $\mathcal{F}$ has compact leaves and $M/\mathcal{F}$ is Hausdorff, it is known that every leaf of $\mathcal{F}$ has finite holonomy (see \cite[Thm. 1.4]{millett}). By Reeb's stability theorem, there exists hence a saturated tubular neighborhood $\mathrm{Tub}(L)\supset L$ of $L$ and a smooth retraction $\chi:\mathrm{Tub}(L)\rightarrow L$, such that $\chi$ restricted to each leaf inside $\mathrm{Tub}(L)$ is a finite covering map. In fact, $\mathrm{Tub}(L)$ can be constructed as the saturation of a local transversal $S_{\pi(p)}$ at $\pi(p)$, taken small enough so that the holonomy group $\mathrm{Hol}_x(L)$
acts on it by diffeomorphisms, which is possible since $\mathrm{Hol}_x(L)$ is finite. In this setting, $\mathrm{Tub}(L)$ is foliated-diffeomorphic to the suspension of the holonomy homomorphism $\pi_1(L)\to \mathrm{Hol}_x(L)$ (see, e.g., \cite[Thm. 2.6]{alex4} for details). Shrinking $S_{\pi(p)}$ if necessary, we can further suppose that $\mathrm{Tub}(L)$ is the union of $\pi$-evenly covered open sets that locally trivialize $\f$. Notice that this ensures that $\mathrm{Tub}(L'):=\pi^{-1}(\mathrm{Tub}(L))$ is a tubular neighborhood of $L'$ with similar properties. In particular, there is a unique smooth map $\chi': \mathrm{Tub}(L') \to L'$ such that the diagram 
$$
\xymatrix{
\mathrm{Tub}(L')\ar[rr]^-{\pi|_{\mathrm{Tub}(L')}}\ar[d]_-{\chi'} & & \mathrm{Tub}(L)\ar[d]^-{\chi}\\
L^\prime\ar[rr]_-{\pi|_{L^\prime}}& & L
}$$
commutes. We therefore conclude that each leaf of $\mathcal{F}^\ast$ inside $\mathrm{Tub}(L')$ covers both $L^\prime$ and $L$.

Now, choose any bundle-like affine structure $(\mathcal{H},\nabla)$ associated with $[\hat{\nabla}]$. Define the map $f$ as the restriction of $\exp^\nabla$ to vectors in $\mathcal{H}$ along $L$, that is $f:=\exp^\nabla_{\mathcal{H}|_L}$. The absence of transverse conjugate points and Remark \ref{transverse-conjugate-points-imply-focal-points} imply that $f$ is a local diffeomorphism. Furthermore, by Propositions \ref{pseudoconvexity-rises} and \ref{disprisonment-rises}, we have that the family $\mathcal{C}_\mathcal{H}$ is pseudo-convex and disprisoning. Since $\mathcal{H}$ is closed and the leaves of $\mathcal{F}$ are compact, we can employ Proposition \ref{pc+disp=proper} to conclude that $f$ is a covering map. Now, consider the universal cover $\varphi: \tilde{L}\rightarrow L$ of $L$, and note that we have a vector bundle $\pi_0:\mathcal{H}|_L\rightarrow L$. If $E$ denotes the pullback bundle $\varphi^*(\mathcal{H}|_L)$, we have that 
$$\xymatrix{
&E\ar@{->}[r]^-{\pi_1}\ar@{->}[d]_-{\pi_2} \ar@{.>}@/_1pc/[dl]_-{\rho} &\tilde{L}\ar@{->}[d]^-{\varphi}\\
M&\mathcal{H}|_L\ar@{->}[r]_-{\pi_0}\ar@{->}[l]^-{f}&L
}$$
commutes, and moreover, that $\pi_2:E\rightarrow\mathcal{H}|_L$ is a covering map, since it has fibers homeomorphic to those of $\varphi$, which are discrete. It follows that $\rho:=f\circ\pi_2:E\rightarrow M$ is a covering map. However, $E$ is a vector bundle over $\tilde{L}$, and $\tilde{L}$ is simply connected, which means that $E$ is simply connected, and therefore it is diffeomorphic to $\tilde{M}$. In fact, given $q\in \rho^{-1}(\pi(p))\cap \tilde{L}$,\footnote{Note that the image of the zero section of $\pi_1:E\rightarrow \tilde{L}$ is diffeomorphic to $\tilde{L}$ and therefore we denote them by the same symbol.} there is a unique diffeomorphism $\Psi:\Tilde{M}\rightarrow E$, such that $\rho\circ\Psi=\pi$ and $\Psi(p)=q$.  Then $\rho(\Psi(L^\prime))=L\implies \Psi(L^\prime)\subset\rho^{-1}(L)$ and, by a connectedness argument, we conclude that $\Psi(L^\prime)\subset \tilde{L}$. By the same token, we see that $\Psi^{-1}(\tilde{L})\subset L^\prime$, which yields that $\Psi$ maps $L^\prime$ diffeomorphically onto $\tilde{L}$.

%This means that $L^\prime\cap\mathcal{W}$ covers its neighboring leaves of $\tilde{\mathcal{F}}|_{\mathcal{W}}$. It follows that not only $L^\prime\cong \tilde{L}$ but also every leaf in a saturated neighborhood $\tilde{\mathcal{V}}$ of $L^\prime$ is diffeomorphic to $\Tilde{L}$, that is, $\tilde{\mathcal{V}}\cong \tilde{L}\times \mathcal{V}$, for some open set $\mathcal{V}\subset \mathbb{R}^q$. We thus conclude that the leaf space $\tilde{M}/\mathcal{F}^*$ is locally Euclidean, and in fact, it has the structure of a locally trivial fibration. Moreover, it is easy to see that $M/\mathcal{F}$ being Hausdorff implies that $\tilde{M}/\mathcal{F}^*$ is Hausdorff as well. Since it is always second-countable, we have that $B:=\tilde{M}/\tilde{\mathcal{F}}$ is a smooth manifold. The projection $\tilde{\pi}:\tilde{M}\rightarrow B$ is thus a submersion.

 Since $L^\prime$ is hence simply connected, it follows that the suspension of $\pi_1(L')\to \mathrm{Hol}_p(L')$ is trivial. Therefore $\mathcal{F}^\ast|_{\mathrm{Tub}(L')}$ is given, say, by the trivial product $\mathrm{Tub}(L')=L'\times S_p$, where $S_p$ is the connected component of $\pi^{-1}(S_{\pi(p)})$ containing $p$. Since $p$ was arbitrary, we thus conclude that the leaf space $\tilde{M}/\mathcal{F}^*$ is locally Euclidean. Moreover, it is easy to see that $M/\mathcal{F}$ being Hausdorff implies that $\tilde{M}/\mathcal{F}^*$ is Hausdorff as well. Since it is always second-countable, we have that $B:=\tilde{M}/\mathcal{F}^\ast$ is a smooth manifold. The projection $\tilde{\pi}:\tilde{M}\rightarrow B$ is thus a submersion.

$$\xymatrix{&\tilde{M}\ar@{->}[rd]^-{\tilde{\pi}}\ar@{->}[ld]_-{\pi}\\
M&&B}$$

Since $\pi$ is a foliated local diffeomorphism and $\mathcal{F}^*$ has the same dimension of $\mathcal{F}$, Remark \ref{push-pull-local-diffeo} allows us to consider the pullback transverse affine connection $\tilde{\nabla}:=\pi^*\hat{\nabla}$ on $(\tilde{M},\mathcal{F}^*)$. Now, employing Remark \ref{push-pull-submersions}, we have that $\tilde{\nabla}$ can be projected via $\tilde{\pi}$, to yield an affine connection $\overline{\nabla}$ on $B$. Note that, since $\pi$ is a local diffeomorphism, the local affine geometry of $(M, \mathcal{F}, [\hat{\nabla}])$ is the same one as the one of $(\tilde{M}, \mathcal{F}^*,[\tilde{\nabla}])$, so that the latter is again without pairs of transverse conjugate points. This, in the light of the discussion which introduced transverse Jacobi fields, yields that $(B, \overline{\nabla})$ does not have any pair of conjugate points. Now, let us see that $(B, \overline{\nabla})$ is pseudoconvex and disprisioning.

%Lemma - If $W\subset M$ is transversely compact, then $\tilde{\pi}(\pi^{-1}(W))\subset B$ is compact. Let $(\mathcal{U}_i)_{i\in I}$ be a cover for $\tilde{\pi}(\pi^{-1}(W))$. Then, $(\tilde{\pi}^{-1}(\mathcal{U}_i))_{i\in I}$ is a cover for $\pi^{-1}(W)$ by sets which are saturated by the leaves of $\tilde{\mathcal{F}}$. Then, $(\pi(\tilde{\pi}^{-1}(\mathcal{U}_i)))_{i\in I}$ is a cover for $\pi(\pi^{-1}(W))=W$ and each of these open sets is saturated by the leaves of $\mathcal{F}$. By the transverse compactness of $W$, we have a finite number of $i\in I$, say $i_1,\ldots, i_k$, such that $W\subset \pi(\tilde{\pi}^{-1}(\mathcal{U}_{i_1}))\cup\ldots\cup\pi(\tilde{\pi}^{-1}(\mathcal{U}_{i_k}))$. Then, $$\pi^{-1}(W)\subset \pi^{-1}\left(\pi(\tilde{\pi}^{-1}(\mathcal{U}_{i_1}))\cup\ldots\cup\pi(\tilde{\pi}^{-1}(\mathcal{U}_{i_k}))\right)=\pi^{-1}(\pi(\tilde{\pi}^{-1}(\mathcal{U}_{i_1})))\cup\ldots\cup\pi^{-1}(\pi(\tilde{\pi}^{-1}(\mathcal{U}_{i_k})))$$

\begin{claim}
If $\pi$ is proper, then it is transversely proper.
\end{claim}

Indeed, let $K\subset M$ be transversely compact, and $(\mathcal{U}_\lambda)_{\lambda\in \Lambda}$ be a cover of $\pi^{-1}(K)$ by open sets saturated by the leaves of $\mathcal{F}^*$. Since $\pi$ is proper, it has a finite number of sheets, say $m$. For each $p\in K$, let $\mathcal{V}_p$ be an evenly covered neighborhood of $p$ such that each connected component of $\pi^{-1}(\mathcal{V}_p)$ is contained in some $ \mathcal{U}_i$. If $\hat{\mathcal{V}}_p$ denotes the $\mathcal{F}$-saturation of $\mathcal{V}_p$, the family of all $\hat{\mathcal{V}}_p$ is a saturated cover for $K$, which means that there exists a subfamily $\hat{\mathcal{V}}_1,\ldots, \hat{\mathcal{V}}_l$ which covers $K$. If $\mathcal{W}_i^1,\ldots, \mathcal{W}_i^{m_i}$ ($m_i\leq m$) are the connected components of $\pi^{-1}(\hat{\mathcal{V}}_i)$, then it is clear that each of these sets is $\mathcal{F}^*$-saturated and is contained in some $\mathcal{U}_\lambda$. It follows that the family $\{\mathcal{W}_i^{j_i}:i=1, \ldots, l; j_i=1\ldots, m_i\}$, is a finite cover of $\pi^{-1}(K)$ which refines $(\mathcal{U}_\lambda)_{\lambda\in \Lambda}$. We thus obtained a finite subcover for $(\mathcal{U}_\lambda)_{\lambda\in \Lambda}$, hence $\pi^{-1}(K)$ is transversely compact, establishing the claim.

Let $K\subset B$ be a compact subset. Then $\tilde{\pi}^{-1}(K)$ is transversely compact in $\tilde{M}$ and so is $\pi(\tilde{\pi}^{-1}(K))$ in $M$. Applying the transverse pseudoconvexity of $(M, \mathcal{F}, [\Hat{\nabla}])$ we obtain a transversely compact $K^*\subset M$ such that every $\hat{\nabla}$-geodesic with endpoints in $\pi(\tilde{\pi}^{-1}(K))$ is contained in $K^*$. Since $\pi$ is proper, by the previous claim it is transversely proper, and thus we have that $\pi^{-1}(K^*)$ is transversely compact, hence $\tilde{\pi}(\pi^{-1}(K^*))$ will be compact in $B$. Let $\gamma:[a,b]\rightarrow B$ be a $\overline{\nabla}$-geodesic in $B$ with $\gamma(a), \gamma(b)\in K$ and suppose that, for some $t_0\in(a,b)$, it verifies $\gamma(t_0)\notin \tilde{\pi}(\pi^{-1}(K^*))$. Now, if $\tilde{\gamma}$ is any lift of $\gamma$ to $\tilde{M}$, by Remark \ref{local-diffeo-transv-geo} it follows that $\tilde{\gamma}$ is a $[\tilde{\nabla}]$-transverse-geodesic, and thus $\pi\circ\tilde{\gamma}$ is a $[\hat{\nabla}]$-transverse-geodesic. On the other hand, $$
\gamma(t_0)\notin \tilde{\pi}(\pi^{-1}(K^*))\implies \tilde{\gamma}(t_0)\notin\pi^{-1}(K^*)\implies \pi\circ\tilde{\gamma}(t_0)\notin K^*,
$$which is a contradiction. Therefore pseudoconvexity holds on $(B,\overline{\nabla})$. For disprisonment, let $\gamma:(a,b)\rightarrow B$ be an inextendible $\overline{\nabla}$-geodesic on $B$ and suppose there is $t_0\in(a,b)$ and a compact $K\subset M$ such that $\gamma[t_0,b)\subset K$. Therefore, if $\tilde{\gamma}$ is any $\tilde{\nabla}$-transverse-geodesic which projects to $\gamma$, then $\tilde{\gamma}[t_0,b)$ is contained in the transversely compact $\tilde{\pi}^{-1}(K)$, and in the same manner, the $\hat{\nabla}$-transverse geodesic $\hat{\gamma}:=\pi(\tilde{\gamma})$ is such that $\hat{\gamma}[t_0,b)$ is contained in the transversely compact $\pi(\tilde{\pi}^{-1}(K))\subset M$. By the transverse disprisonment of $(M,\mathcal{F}, [\hat{\nabla}])$ it follows that $\hat{\gamma}$ must have a right endleaf $L\in\mathcal{F}$. Then, it is easy to see that some leaf $L^\prime\in\mathcal{F}^*$ with $\pi(L^\prime)=L$ will be a right endleaf of $\tilde{\gamma}$. Finally, since $\tilde{\pi}$ is a submersion, it follows that it maps $L^\prime$ to a right endpoint of $\gamma$, contradicting the assumption of its inextendibility.

Finally, Theorem \ref{pc+disp=proper} applied to the family of all geodesics of $(B,\overline{\nabla})$ implies that, for any $x\in B$, $\exp_x$ is proper. Since there are no conjugate points, it is also a local diffeomorphism. We therefore obtain that $\exp_x$ is a covering map. However, since $\tilde{M}$ is simply connected and $\tilde{\pi}$ has connected fibers, it follows from the fibration homotopy sequence
$$\cdots \to \pi_1(\tilde{L}) \to \pi_1(\tilde{M})\to \pi_1(B)\to \pi_0(\tilde{L})\to \cdots$$
that $B$ is simply connected. Thus, $\exp_x$ is the universal cover, meaning that $B$ is contractible, which in turn implies that the fiber bundle $(\Tilde{M}, \tilde{\pi}, B, \tilde{L})$ is trivial, that is, $\tilde{M}=\tilde{L}\times B$.
\end{proof}

We recall that a foliation $(M,\mathcal{F})$ is called \textit{developable} when there is some smooth covering map $p:M^\prime\rightarrow M$ such that the pullback foliation $p^*\mathcal{F}$ on $M^\prime$ is a simple foliation (i.e., given by connected fibers of a submersion). In this case we say that $p^*\mathcal{F}$ is a \textit{development} of $\f$, or that $\f$ \textit{develops} to it. From what we saw in the previous proof, Theorem \ref{thrm princ} can be alternatively restated with this terminology:

\begin{corollary}
Let $(M,\mathcal{F})$ be a foliation whose leaves are all compact and such that $M/\mathcal{F}$ is Hausdorff. Suppose $(M,\mathcal{F})$ is endowed with a transverse affine structure $[\Hat{\nabla}]$ such that the family of all non vertical transverse-geodesics is transversely pseudoconvex, transversely disprisoning and without transverse conjugate points. Then $\mathcal{F}$ develops, on the universal covering $\tilde{M}$ of $M$, to a simple foliation over a contractible manifold $B$.
\end{corollary}

%Another important remark is that, as discussed in detail in a previous work \cite{transverse-diameter}, the coexistence of the hypotheses of Hausdorffness of $M/\f$ and the compactness of the leaves imply (see \cite[Thm. 1.4]{millett}) that $\mathcal{F}$ has \textit{finite} holonomy and therefore $M/\mathcal{F}$ has a natural orbifold structure. However, as pointed out in Remark \ref{suf-not-nec}, the compactness hypothesis in Proposition \ref{pseudoconvexity-rises} is not a necessary condition. Therefore we leave to future research to decide whether one can establish this results in a broader context.

As announced in the Introduction, an interesting and important corollary of the results in this paper is applicable to an \textit{affine orbifold} $(\mathcal{O},\nabla)$, that is, an orbifold endowed with a \textit{cone connection} (for a detailed discussion, see \cite[Section 2.2]{daza}). When dealing with orbifolds, the notion of a covering map is slightly more involved than usual (see \cite[Section 2.3]{caramello3}), since it has to take eventual singularities into account. Even so, we can still define an orbifold $\mathcal{O}$ to be \textit{developable} if there is an orbicover $p:\mathcal{O}^\prime\rightarrow\mathcal{O}$ such that $\mathcal{O}^\prime$ is a manifold. It is clear that being developable in this sense is equivalent as the foliation which realizes it as a leaf space being developable in the sense mentioned above. The definitions of pseudoconvexity and disprisonment for families of geodesics thereon are identical. As mentioned above, the orbifold $\mathcal{O}$ can be realized as a leaf space $M/\mathcal{F}$ for a foliation $\mathcal{F}$ of finite holonomy and compact leaves, and we have a map $\pi:M\rightarrow\mathcal{O}$ which is a submersion in the sense of orbifolds (see, e.g., \cite{caramello3}). A construction entirely analogous to that of Proposition \ref{push-pull-submersions} will show that $\nabla$ induces a transverse affine connection $\hat{\nabla}$ on $M$, and the reasoning of Remark \ref{local-diffeo-transv-geo} can be extended to orbifolds to yield that $\pi$ projects $\hat{\nabla}$-transverse-geodesics onto $\nabla$-geodesics. One can thus prove that the pseudoconvexity and disprisonment of $\nabla$ implies the transverse pseudoconvexivity and transverse disprisonment of $\hat{\nabla}$. Analogously to the theory developed in Section 5, we have that $\hat{\nabla}$-transverse Jacobi fields are mapped by $d\pi$ into Jacobi fields of $(\mathcal{O},\nabla)$ and, therefore, the absence of conjugate points in $(\mathcal{O},\nabla)$ imply the absence of transverse conjugate points in $(M,\mathcal{F},[\hat{\nabla}])$. We have thus established Corollary \ref{corx}, whose statement we repeat here:  

\begin{corollary}[Affine Hadamard theorem for orbifolds]
Let $(\mathcal{O},\nabla)$ be an affine orbifold without conjugate points and such that the family of its geodesics is pseudoconvex and disprisoning. Then $\mathcal{O}$ is a good (developable) orbifold and whose universal covering is a contractible manifold.
\end{corollary}

\section*{Acknowledgements}
IPCS is partially supported by the project PID2020-118452GBI00 of the Spanish government. HPM was financed in part by the Coordenação de Aperfeiçoamento de Pessoal de Nível Superior - Brasil (CAPES) - Finance Code 001.

\end{document}